\documentclass[final,onefignum,onetabnum]{siamart171218}

\usepackage{latexsym, amssymb, enumerate, amsmath}
\usepackage{defs}
\usepackage{fonts}
\usepackage{epsfig}

\usepackage{epstopdf}
\usepackage{caption}
\usepackage{ulem}
\usepackage{breqn}
\usepackage[bbgreekl]{mathbbol}
\usepackage{bbm}

\usepackage{subfig}
\newtheorem{remark}{Remark}[section]
\newtheorem{assumption}{Assumption}
\newtheorem{assumptionPrime}{Assumption}

\newtheorem{cuslemma}{Lemma}[section]

\usepackage{etoolbox}
\patchcmd{\SetTagPlusEndMark}{$}{}{}{}
\patchcmd{\SetTagPlusEndMark}{$}{}{}{}

\makeatletter %
\newtheoremstyle{nospacethm}%
{\item[\hskip\labelsep \theorem@headerfont ##1 ##2\theorem@separator]}%
\makeatother

\theoremstyle{nospacethm}
\newtheorem{reqbase}{Requirement}
\renewcommand{\thereqbase}{BI\arabic{reqbase}}
\crefname{reqbase}{Requirement}{Requirements}

\usepackage{lineno}

\usepackage{url}
\usepackage{cite}
\usepackage{multirow} 
\usepackage{enumitem}
\usepackage{mathrsfs}

\usepackage{algorithmicx}
\usepackage{xifthen}
\usepackage{needspace}
\usepackage{algorithm}
\usepackage{algpseudocode}
\usepackage{textcomp}
\usepackage{xfrac}
\usepackage{upgreek}

\graphicspath{{./figures/}}

\newcounter{myalg}
  {
    \needspace{2\baselineskip}
    \noindent \rule{\linewidth}{1pt} \endgraf
    \refstepcounter{myalg}
    \centering \textsc{Algorithm}~\themyalg%
    \ifthenelse{\isempty{#1}}{}{:\ #1}
  }{
  \noindent \rule{\linewidth}{1pt}
  }%

\title{
An Infeasible-Start Framework for
Convex Quadratic Optimization, 
with Application to Constraint-Reduced
Interior-Point Methods
\thanks
	{This manuscript has been authored, in part, by UT-Battelle, LLC, under Contract No. DE-AC0500OR22725 with the U.S. Department of Energy. The United States Government retains and the publisher, by accepting the article for publication, acknowledges that the United States Government retains a non-exclusive, paid-up, irrevocable, world-wide license to publish or reproduce the published form of this manuscript, or allow others to do so, for the United States Government purposes. The Department of Energy will provide public access to these results of federally sponsored research in accordance with the DOE Public Access Plan (\texttt{http://energy.gov/downloads/doe-public-access-plan}).}}
\date{\today}
\author{M. Paul Laiu%
\thanks{Computational and Applied Mathematics Group,
	Computer Science and Mathematics Division,
	Oak Ridge National Laboratory,
	Oak Ridge, TN 37831 USA,
(\texttt{laiump@ornl.gov}).}
\and
Andr{\'e} L.~Tits%
\thanks{Department of Electrical and Computer Engineering
  \& Institute for Systems Research,
  University of Maryland
  College Park, MD 20742 USA,
(\texttt{andre@umd.edu}).}
}

\begin{document}

\maketitle

\begin{abstract}
A framework is proposed for solving general convex 
quadratic programs (CQPs) 
from an infeasible starting point by invoking an existing 
{\em feasible-start} algorithm tailored for
{\em inequality}-constrained CQPs.  
The central tool is an exact penalty function scheme equipped with a 
penalty-parameter updating rule.  The feasible-start algorithm
merely has to satisfy certain general requirements, 
and so is the updating rule.  
Under mild assumptions, the framework is proved to converge
on CQPs with both inequality and equality constraints and, 
at a negligible additional cost per iteration,
produces an infeasibility certificate, together with a 
feasible point for an (approximately) $\ell_1$-least relaxed 
feasible problem when the given problem does not have a feasible solution.
The framework is applied to a feasible-start constraint-reduced
interior-point algorithm previously proved to be highly performant
on problems with many more constraints than variables (``imbalanced'').  
Numerical comparison with popular codes (SDPT3, SeDuMi, MOSEK) is 
reported on both randomly generated problems and support-vector 
machine classifier training problems.  
The results show that the former typically outperforms
the latter on imbalanced problems.
\end{abstract}

\begin{keywords}
convex quadratic programming, 
infeasible start,
infeasibility certificate,
constraint reduction, 
interior point,
support-vector machine
\end{keywords}

\begin{AMS}
65K05, 90C05, 90C06, 90C20, 90C51
\end{AMS}

\section{Introduction}
\label{sec:CR_MPC}

Consider a convex quadratic program (CQP)
\begin{equation}
\minimize_{\bx\in\bbR^n} \:
\bff(\bx):=\frac{1}{2}\bx^T\bH\bx+\bc^T\bx \mbox{~s.t.~} \:
\bA\bx \geq \bb,\: \bC\bx=\bd,
\tag{P}
\label{eq:primal_cqp}
\end{equation}
where $\bx\in\bbR^n$ is the vector of optimization variables,
$\bff\colon\bbR^n\to\bbR$
the objective function, with $\bc\in\bbR^n$, 
$\bH\in\bbR^{n\times n}$ symmetric positive semi-definite;
and where
$\bA\in\bbR^{m\times n}$ and $\bb\in\bbR^m$,
$\bC\in\bbR^{p\times n}$ and $\bd\in\bbR^p$,
with $n\geq p$ and $m+p>0$.
Here and elsewhere all inequalities ($\geq$, $\leq$, $>$, $<$)
are meant component-wise.

Most available algorithms for solving such problems belong either to the
interior-point family or to the simplex-like family.  While the most
popular interior-point algorithms do not 
require that an initial feasible point be provided,
simplex algorithms  do:
such feasible points, when not readily available, are typically
obtained by solving an auxiliary linear optimization problem 
(``phase~1'').  Like simplex algorithms, recently proposed 
``constraint-reduced'' interior-point algorithms,
the latest of which (see, e.g.,~\cite{LaiuTits19}) 
were observed to often largely outperform other approaches when the problem 
at hand is severely ``imbalanced'' 
(i.e., with most inequality constraints being inactive at the solution; 
e.g., $m\gg n-p$), 
do require a 
primal-feasible initial point.  While a two-phase approach 
could again be employed here, 
an important drawback of two-phase
approaches is that, in the first phase, the objective function
is altogether ignored, leading to likely computational waste.\footnote{Note
however that, in the context of feasible-direction methods for general
nonlinear optimization problems, methods that craftily combine the two
phases have been known for decades; see~\cite{Polak1979,
Polak1991}.}

Motivated by the above, the aim of the present paper is to propose
an exact-penalty-function-based framework that ``transforms'' 
an available primal-feasible algorithm
into one that accommodates infeasible starts.   While the original
intent was restricted to ``infeasibilizing'' algorithm CR-MPC\footnote{A
constraint-reduced version of Mehrotra's Predictor Corrector \cite{Mehrotra-1992}.}
of~\cite{LaiuTits19}, it has broaden to the development of a scheme
that invokes an essentially arbitrary feasible-start method.

The central component of the framework is an augmented version
of~\eqref{eq:primal_cqp} that involves a vector of relaxation
variables
and an exact penalty function.  Exact penalty functions (i.e.,
penalty functions for which a threshold exists---but is unknown at
the outset---such that, when the penalty parameter exceeds that
threshold, solutions of the penalized problem also solve the
original constrained problem) have been extensively used for
many decades in nonlinear optimization, especially since the
seminal work of A.R.~Conn~\cite{Conn1973};
see, e.g., \cite{ColemanConn82,Conn1998,Antczak09,HassanBaharum19}.
While the adaptive
selection of the penalty parameter is often heuristic, in some
contexts, authors have proposed formal adaptation rules that
guarantee that an appropriate value of the parameter will eventually
be obtained and will be kept for the remainder of the solution
process; this goes back several decades (e.g.,~\cite{Mayne1976}
as well as, in the context of augmented Lagrangian, ~\cite{Polak1980,
DiPilloEtal92})
and also includes more recent work such as~\cite{TWBUL-2003,Byrd2008}.

While, originally, the intent of exact penalty functions was
to turn a constrained-optimization problem into an unconstrained one,
such tool has also been used to eliminate equality constraints when only
an inequality-constraint algorithm is available, specifically, by 
replacing in each scalar equality the ``='' sign with `'$\geq$'' and
penalizing positive deviations from equality; 
see~\cite{Mayne1976,TWBUL-2003}.   More recently, in~\cite{He-Thesis,HT12},
exact penalty functions have been used for yet another purpose: allowing 
algorithms that require a feasible initial point to accept infeasible
initial points.  As pointed out above, this is the focus of 
the present work.

Use of penalty functions in the solution of linear or convex-quadratic
optimization problems has been scarcer than their use in nonlinear
optimization, for obvious reasons: powerful
methods have long existed (starting with the original simplex method
for linear optimization seven decades ago) for the solution of such 
problems and there was no perceived need to resort to such tool.
Exceptions include the use of an exact penalty method
for warmstarting interior-point methods~\cite{Benson2007}
and the ``big $M$'' approach (where the penalty parameter 
is ``large'' but fixed) considered in~\cite[section~4.3]{HungerlanderThesis}.
Also, as mentioned above, such need does arise in the context 
of constraint-reduced interior-point methods.  An exact penalty
function scheme was thus used in~\cite{He-Thesis,HT12} in the context 
of a specific constraint-reduced algorithm for 
inequality-constrained linear~\cite{HT12}, then convex 
quadratic~\cite{He-Thesis}, optimization.

In the present paper, a rather general framework is proposed,
analyzed, and numerically tested, for the solution of a CQP,
starting from a primal-infeasible point, that invokes 
an iteration of a rather arbitrary user-provided feasible-start
CQP solver, referred to below as ``base iteration".
The key contributions are 
as follows.   First the approach introduced in~\cite{He-Thesis}
is generalized to apply to a general class of feasible-start
base iterations (as opposed to, merely, 
a specific version of a constraint-reduced scheme), and to offer 
broad freedom in
the choice of a penalty-parameter updating rule; the base iteration and
the updating rule are merely required to satisfy certain general
specifications.
Second, the framework is then extended to solve
problems that include equality constraints without 
destroying any existing sparsity.
Third, it is shown that, 
at a negligible additional cost per iteration,
when the CQP is primal-infeasible, a 
certificate of infeasibility is produced.  Finally, promising 
numerical results are obtained, with the algorithm of~\cite{LaiuTits19}
as base iteration, in comparison
with those obtained with popular schemes.

The paper is organized as follows. In section~\ref{sec:framework},
the framework is outlined, and requirements to be satisfied by
the base iteration and the penalty-parameter updating rule are introduced.
Section~\ref{sect:analysis} is devoted to the convergence analysis,
under the assumption that the requirements specified in
section~\ref{sec:framework} are satisfied.  For sake of simplicity 
of exposition, sections~\ref{sec:framework}
and~\ref{sect:analysis} deal with purely inequality-constrained
problems, i.e., $p=0$.  Extension to the general problem is dealt 
with in section~\ref{sec:eq-constrained}.  In section~\ref{sec:infeasibility},
issuance of an infeasibility certificate in case \eqref{eq:primal_cqp} is infeasible
is investigated.  
Section~\ref{sec:opt_num_results} introduces a penalty-parameter update that
satisfies the required specifications, discusses implementation details, and reports numerical
results on randomly generated problems and support-vector machine training problems 
with comparison to popular optimization solvers.
Concluding remarks are given in section~\ref{sec:conclusion}.

The notation is mostly standard.  
In particular, consistent with the interior-point literature, given
a vector $\bv$, the associated matrix $\diag(v_i)$ is denoted by the 
corresponding capital letter $V$.
We use $\|\cdot\|$ to denote
an arbitrary norm, possibly different in each instance that it
is being used; of course, $\|\cdot\|_\infty$, $\|\cdot\|_1$,
and $\|\cdot\|_2$ are specific.  
The matrix norms are the respective induced norms.
The Matlab notation ($[A~B;C~D]$, $[\bu;\bv]$) is used for
block matrices and vector concatenation.

Before proceeding, we state here two assumptions on
problem~\eqref{eq:primal_cqp}, which will be in force throughout---with the exception
of section~\ref{sec:infeasibility}, as duly noted there.  Recall 
(e.g.,~\cite[Propositions 2.1--2.2]{MonteiroAdler89b}) that
if the dual of a CQP is feasible then the CQP
is bounded, and that if the CQP is feasible and bounded 
then it has an optimal solution and its dual is feasible.

\begin{assumption}
\label{as:PD}
\eqref{eq:primal_cqp} is strictly feasible and so is its dual,
and \eqref{eq:primal_cqp}'s (nonempty) optimal solution set is 
bounded.\footnote{Only Proposition~\ref{prop:constantPhi} invokes
boundedness of the primal optimal solution set. It is not clear
at this point whether such assumption is necessary indeed.  In any
case, if it is, it of course can be achieved by imposing large bounds 
to the components of $\bx$.}
\end{assumption}

\begin{assumption}
\label{as:LI}
$\bC$ has full (row) rank and $[\bH;\bA;\bC]$ has full (column) 
rank.
\end{assumption}

\section{A Framework for Accommodating Infeasible Starts}
\label{sec:framework}

\subsection{General Idea}
\label{subsec:generalidea}

Suppose a feasible-start base iteration is available toward
solving~\eqref{eq:primal_cqp} with $p=0$
and suppose moreover that
applying such iteration repeatedly on \eqref{eq:primal_cqp} 
produces a sequence of 
feasible iterates that enjoys certain additional properties (to be 
specified in section~\ref{sec:req-BA} below). 
It is suggested in~\cite{LaiuTits19}, in the context of a 
``constraint-reduced'' primal-dual interior-point method
that requires an initial primal-feasible point,
that an extension to handle problems for which a 
primal-feasible initial point is {\em not} available 
can be obtained by involving the following surrogate 
primal--dual pair,\footnote{{An $\ell_\infty$ penalty function can 
be substituted for this $\ell_1$ penalty function with minor adjustments: 
see \cite{He-Thesis,HT12} for details.}}
for which a 
primal-feasible 
point $(\bx,\bz)$ is 
readily available:

\begin{align}
\minimize_{\bx\in\bbR^n,\,\bz\in\bbR^m} \: \bff(\bx)&+\varphi\bone^T\bz
\quad \mbox{s.t.~} \:  \bA\bx +\bz \geq \bb,\: \bz\geq\bzero\:,
\label{eq:one}\\
\maximize_{\bx\in\bbR^n,\,\bspi\in\bbR^m,\bsxi\in\bbR^m}\:
\bspsi(\bx,\bspi)&
\mbox{~s.t.~} \:    \bH\bx+\bc-\bA^T\bspi=\bzero\:,
\bspi+\bsxi = \varphi\bone, (\bspi,\bsxi)\geq\bzero,
\end{align}
with
$
\bspsi(\bx,\bspi):=-\frac{1}{2}\bx^T\bH\bx+\bb^T\bspi,
$
where $\varphi>0$ is a penalty parameter.
Equivalently,
\begin{align}
&\minimize_{\bbx\in\bbR^{n+m}} \: {\bbf}_\varphi(\bbx):= \frac{1}{2} \bbx^T\bbH\bbx + \bbc_\varphi^T\bbx
\quad \mbox{s.t.~} \:  \bbA\bbx \geq \bbb\:,
\tag{P$_\varphi$}
\label{eq:primal_cqp_rho}\\
\maximize_{(\bbx,\bblambda)\in\bbR^{(n+m)+2m}}&\:
\bbpsi(\bbx,\bblambda):=-\frac{1}{2}\bbx^T\bbH\bbx+\bbb^T\bblambda
\quad\mbox{s.t.~} \: \bbH\bbx+\bbc_\varphi-\bbA^T\bblambda=\bzero,
~\bblambda\geq\bzero\:,
\tag{D$_\varphi$}
\label{eq:dual_cqp_rho}
\end{align}
where $\bbx:=[\bx;\bz]$, $\bblambda:=[\bspi;\bsxi]$, $\bbc_\varphi:=[\bc;\varphi\bone]$, $\bbb:=[\bb;\bzero]$,
\begin{equation}
\bbH := \left[\begin{array}{cc}
\bH & \bzero\\ \bzero & \bzero
\end{array}\right]\:,\quand
\bbA := \left[\begin{array}{cc}
\bA & \bI\\ \bzero & \bI
\end{array}\right]\:.
\end{equation} 

\noindent
Necessary and sufficient conditions for $(\bx,\bz,\bspi,\bsxi)$ to
solve~\eqref{eq:primal_cqp_rho}--\eqref{eq:dual_cqp_rho} are given by
\begin{equation}
\bH\bx + \bc - \bA^T\bspi = \bzero,~
\bspi+\bsxi=\varphi\bone,~
\bS\bspi=\bzero,~
\bZ\bsxi=\bzero,~
(\bs,\bz,\bspi,\bsxi)\geq\bzero.
\end{equation}
where $\bs:=\bA\bx+\bz-\bb$; 
equivalently,
\begin{equation}
\bbH\bbx + \bbc_\varphi - \bbA^T \bblambda = \bzero,~
\bbS\bblambda=\bzero,~ (\bbs,\bblambda)\geq\bzero,
\end{equation}
where $\bbs:=[\bs;\bz]$.

The rationale for introducing~\eqref{eq:primal_cqp_rho} is as follows. 
The penalty function penalizes positive values of the components
of $\bz$ while the $\bz\geq\bzero$  constraints in~\eqref{eq:one}
prevent negative values.
Hence, since the $\ell_1$ penalty function is known to be exact, 
if the solution set of~\eqref{eq:primal_cqp} is nonempty
(implying that the solution set of~\eqref{eq:primal_cqp_rho} is nonempty
for $\varphi$ sufficiently large), 
for $\varphi$ above a certain threshold, every solution 
of~\eqref{eq:primal_cqp_rho}
will be of the form $(\bx^*,\bzero)$, with $\bx^*$ 
a solution of~\eqref{eq:primal_cqp} (see \cref{lem:phi>lambda} below). 
On the other hand, for given $\varphi$, 
\eqref{eq:primal_cqp_rho}
(with feasible initial $(\bx,\bz)$) can 
be tackled by repeated application of the base iteration.
The idea is then to augment the base iteration with a
penalty-parameter updating scheme to bring $\varphi$ above such threshold.
One such scheme was proposed in~\cite{He-Thesis,HT12} (again in the 
context of a specific constraint-reduced algorithm).
Problem~\eqref{eq:primal_cqp_rho} enjoys the following properties
to be invoked in the analysis.

\renewcommand{\thecuslemma}{2.0}
\begin{cuslemma}
\label{lem:Pvarphi}
Given $\varphi>0$, \eqref{eq:primal_cqp_rho} is strictly feasible.  Further, for
$\varphi>0$ large enough, \eqref{eq:primal_cqp_rho} is bounded, i.e., has a 
nonempty solution set.  
Finally, for $\varphi>0$, given any $\rho>0$ and $\alpha\in\bbR$, the set 
$\cS:=\{(\bx,\bz)\in\cF_\alpha: \|\bz\|\leq\rho\}$ is bounded,
where
$\cF_\alpha:=\{\bbx:\bz\geq\bzero,\bA\bx+\bz\geq\bb,\bbf_\varphi(\bbx)\leq\alpha\}$.

\end{cuslemma}
\begin{proof}
First, trivially, given any $\bx$, there exists $\bz$ with large enough components such that $(\bx,\bz)$ is strictly feasible.  
Next, boundedness of \eqref{eq:primal_cqp_rho}
for $\varphi$ large enough follows from feasibility 
of \eqref{eq:primal_cqp_rho}
and feasibility of \eqref{eq:dual_cqp_rho} for $\varphi$ large enough, 
where the
latter follows from \cref{as:PD},
since the only difference between the dual of~\eqref{eq:primal_cqp}
and~\eqref{eq:dual_cqp_rho} is the constraint $\bspi+\bsxi = \varphi\bone$,
with $\bsxi\geq\bzero$, in the latter.
As for the third claim, proceeding by contradiction, suppose 
that $\cS$ is unbounded.  
Then $\cS$ must contain a nontrivial recession (translated) cone, 
i.e., (since $\bz$ is ``bounded in $\cS$'',)
there exists a ($\varphi$-dependent) direction $\bv\not=\bzero$ 
such that $\bH\bv=\bzero$, $\bc^T\bv\leq 0$, and $\bA\bv\geq\bzero$.
If $\bc^T\bv = 0$, this contradicts to boundedness of the optimal solution set of \eqref{eq:primal_cqp} (\cref{as:PD}). 
On the other hand, if $\bc^T\bv < 0$, this contradicts to the assumption that the optimal solution set of \eqref{eq:primal_cqp} is 
nonempty (again, \cref{as:PD}).
\end{proof}

\subsection{Proposed Framework}
\subsubsection{Master Algorithm}
\label{ProposedAlgorithmStructure}
Given a base iteration and a penalty-parameter updating rule, the
overall algorithm 
for solving~\eqref{eq:primal_cqp} 
starting from a potentially infeasible point 
proceeds as follows.
Here, 
\texttt{Var$_{\texttt{BI}}$} collects all internal base-iteration variables
that are not listed explicitly, and
the parenthetic ($\bblambda^{k+1}$) indicates that 
$\bblambda^{k+1}$ may or
may not be produced by the base iteration; if it is not, an 
appropriate quantity must be generated outside the base iteration,
for input into the penalty-parameter update. 

\medskip\noindent
{\bf Master Algorithm}\\ \nopagebreak
\noindent
{\bf Parameters:}  Parameters of the base iteration
and of the penalty-parameter update. \\ \nopagebreak
{\bf Initialization:}  $\bx^0\in\bbR^n$, $\bz^0\in\bbR^n$,
satisfying $\bz^0{>}\bzero$ and $\bA\bx^0+\bz^0{>}\bb$; 
$\bs^0:=\bA\bx^0+\bz^0-\bb  \:({>}\bzero)$;%
\footnote{%
When the base iteration satisfies \cref{reqB:descent}~(i) (see section~\ref{sec:req-BA}), 
the Master Algorithm can take non-strict feasible initial point, i.e., $\bz^0\geq\bzero$ and $\bA\bx^0+\bz^0\geq\bb$.
} 
$\dv^0\geq\bzero$;
$\varphi_0>0$;  $k:=0$; \texttt{Var$_{\texttt{BI}}$}
\\
{\bf Iteration $k$}: \\
\mbox{}~~~~If $\dv^k\geq\bzero$ is not available, provide a (nonnegative)
estimate thereof; see section~\ref{sec:req-BA}. \\
\mbox{}~~~~{\bf If user-provided stopping criterion is satisfied, stop.} \\
\mbox{}~~~~{\bf Penalty-parameter update}: \\
\mbox{}~~~~\mbox{}~  {\em Input:}  $\varphi_k>0$; $\bx^{k}$, $\bz^{k}$, $\dv^{k}$\\
\mbox{}~~~~\mbox{}~ {\em Output:} 
$\varphi_{k+1}\geq\varphi_k$.\footnote{A variant, for which does the 
analysis of sections 3 to 5 still applies, would be for the penalty-parameter 
update to output $\bx^{k^\prime}$ instead of 
$\bx^k$, with $k^\prime={\rm argmin}_{\ell\in\{1,\ldots,k\}}
\bff(\bx^\ell)+\varphi_{\ell}\bone^T\bz^\ell$.  Numerical tests 
showed no discernible advantage from adopting this variant though.}
\\
\mbox{}~~~~{\bf Base iteration (applied to
$(\textrm{P}_{\varphi_{k+1}})$--$(\textrm{D}_{\varphi_{k+1}})$) }\\
\mbox{}~~~~\mbox{}~  {\em Input:} $\bbx^k:=[\bx^k;\bz^k]$, $\bbs^k:=[\bs^k;\bz^k]$; \texttt{Var$_{\texttt{BI}}$}\\
\mbox{}~~~~\mbox{}~  {\em Output}: $[\bx^{k+1};\bz^{k+1}]:=\bbx^{k+1}$, 
$\bs^{k+1}:=\bA\bx^{k+1}+\bz^{k+1}-\bb$, 
$(\bblambda^{k+1}\geq\bzero)$; \texttt{Var$_{\texttt{BI}}$}
\\
\mbox{}~~~~{\bf If $\bbf_{\varphi_{k+1}}(\bbx^{k+1})>\bbf_{\varphi_{k+1}}(\bbx^{k})$, \\
\mbox{}~~~~\mbox{}~
set $\bx^k:=\bx^{k+1}$, $\bs^k:=\bs^{k+1}$, and $\bz^k:=\bz^{k+1}$;
go back to Base iteration.}\\
\mbox{}~~~~{\bf Otherwise, \\ 
\mbox{}~~~~\mbox{}~~~~set $k:=k+1$ and go to Iteration $k$}.

\begin{remark}
\label{rem:monotone}
Note that, regardless of whether or not the base iteration enforces
monotone decrease of the objective function, the sequence ``seen'' by
the penalty-parameter update does enjoy such property, i.e., upon entry 
into the penalty-parameter
update, $\bff(\bx^k)+\varphi_k\bz^k \leq \bff(\bx^{k-1})+\varphi_k\bz^{k-1}$
for all $k\geq1$.  
Such monotone decrease is key to~\cref{lem:x_z_bounded}
below, on which the convergence analysis relies.
\end{remark}

In the remainder of section~\ref{sec:framework}, we consider requirements to 
be imposed on the base iteration and on the penalty-parameter update; 
in section~\ref{sect:analysis}
we will prove that, when these requirements are satisfied, 
the penalty parameter $\varphi_k$ is eventually constant,
and the primal iteration $\bx^k$ converges to
the optimal solution set of~\eqref{eq:primal_cqp}.
\subsubsection{Requirements for the base iteration}
\label{sec:req-BA}

When the base iteration is applied repeatedly toward solving
a CQP of the form
\begin{equation}
\minimize_{x\in\bbR^n} \:
f(x):=\frac{1}{2}x^THx+c^Tx \mbox{~~s.t.~} \:
Ax \geq b,
\label{eq:primal_cqp_generic}
\end{equation}
(with any stopping criterion turned off,)
it must construct a primal sequence $\{x^\ell\}$ which, together with some 
``dual'' sequence $\{\lambda^\ell\}$
(possibly constructed by the base iteration),
with $\lambda^\ell\geq\bzero$ for all $\ell$,
satisfies the following requirements of feasibility, eventual descent
(\cref{reqB:eventualdescent} guarantees that
$\hat{K}:=\{k>0\, \colon \, f(x^k)\leq f(x^\ell),\, \forall\ell<k\}$
is an infinite index set), 
and---when the descending primal subsequence alluded to above is 
bounded---convergence to the optimal solution set.

\begin{reqbase}\label{reqB:descent}
The base iteration satisfies at least one of the following two properties:
(i)	Given $x^\ell$ primal feasible, $x^{\ell+1}$ is primal feasible. 
(ii) Given $x^\ell$ primal strictly feasible, $x^{\ell+1}$ is primal strictly feasible.
\end{reqbase}

\begin{reqbase}\label{reqB:eventualdescent}
         Given $\ell_0>0$, there exists $\ell>\ell_0$ such that
         $f(x^\ell)\leq f(x^{\ell_0})$.
\end{reqbase}

\begin{reqbase}\label{reqB:cvgce}
If $\{x^k\}_{k\in \hat{K}}$ is bounded, then 
\[
\max\{\|S^k\lambda^k\|,\,
\|Hx^k +c-A^T\lambda^k\| \} \to 0 \,\text{ on }\, \hat{K}\:.
\] 
where $S^k:=\diag(Ax^k-b)$.
\end{reqbase}

\subsubsection{Example: An Infeasible-Start CR-MPC Algorithm}
\label{sec:infeasible_CRMPC}
In~\cite{LaiuTits19}, a constraint-reduced interior-point algorithm
dubbed CR-MPC is
proposed to tackle CQPs for which a strictly primal-feasible 
initial point $\bx$ is available and no equality constraints are
present.
CR-MPC does produce an appropriate $\bblambda$ sequence; 
specifically,
$\bblambda^0$ in ``Initialization'' of the
Master Algorithm is arbitrary and, for $k=0,1,\ldots$, $\bblambda^{k+1}$ 
in the ``Output'' line of the base iteration in the Master Algorithm
is assigned the value $[\tilde\bblambda^+]_+$, where $\tilde\bblambda^+$
is as generated in Step~8 by the $k$th run of iteration CR-MPC.
Here we show that under \cref{as:PD},
iteration CR-MPC satisfies the Requirements~BI in section~\ref{sec:req-BA}.

Because iteration CR-MPC is a primal-strictly-feasible iteration with
monotone decrease of the objective function, 
\cref{reqB:descent,reqB:eventualdescent} are trivially satisfied.   
As for \cref{reqB:cvgce}, it follows from parts~(i) 
and~(iv) of Theorem~1 of~\cite{LaiuTits19}
that it is also satisfied by CR-MPC, provided
that the Assumptions~1 and~2 of~\cite{LaiuTits19} are met by \eqref{eq:primal_cqp_rho}.
Assumption~1 of~\cite{LaiuTits19} requires that \eqref{eq:primal_cqp_rho} be 
strictly feasible, be bounded, and have a bounded solution set.  
The first property is established by \cref{lem:Pvarphi} under \cref{as:PD} 
of the present paper.
The second and third ones are invoked only in Lemma~5 of~\cite{LaiuTits19} 
(see the sentence immediately preceding that lemma) in proving boundedness
of the primal sequence.  
Since, boundedness of $\{\bx^k\}$ is assumed in \cref{reqB:cvgce},
the second and third properties in Assumption~1 of~\cite{LaiuTits19} are not
necessary.
As for
Assumption~2 of~\cite{LaiuTits19} (linear independence of the gradients 
of active constraints
at stationary points), when applied to \eqref{eq:primal_cqp_rho}, it
amounts to requiring linear independence, for all $\bx\in\bbR^n$, 
of $\{\ba_i:\ba_i^T\bx\leq b_i\}$. Accordingly, in order to cover
Assumption~2 of~\cite{LaiuTits19}, we append here a third
assumption to our list; 
{\em it is in force in the present
subsection only}.

\begin{assumption}%
\footnote{While the authors of~\cite{LaiuTits19} (who
are also the authors of the present paper) were not able to do away with
such linear-independence assumption in proving that Theorem~1 of 
that paper holds, intuition and extensive numerical testing 
suggest that \cref{as:LI3} can be dropped.}
For all $\bx\in\bbR^n$, $\{\ba_i:\ba_i^T\bx\leq b_i\}$ is a linearly
independent set.
\label{as:LI3}
\end{assumption}

\noindent
(Note that iteration CR-MPC
 enforces descent of the objective function, so that the
``Otherwise'' exit of the ``If'' test in the Master Algorithm is
always selected.)

\subsubsection{Requirements for the penalty-parameter update}

The penalty-parameter update
has a dual purpose.  First, see to it
that $\varphi_k$ (rapidly) achieves a value sufficient for
\textup{(P$_{\varphi_k}$)} to have a nonempty solution set.
Second, further see to it that such value is high enough that
solutions to \textup{(P$_{\varphi_k}$)} are solutions to
the original problem.  Existence of a threshold insuring the
latter is indeed guaranteed by the ``exact'' character of the
penalty function in~\textup{(P$_{\varphi_k}$)}.  It is
desirable that $\varphi_k$ reach an adequate value rapidly
because of course, every time $\varphi_k$ is updated, the
solution process is disrupted.

In view of~\cref{lem:phi>lambda} below, 
the first three requirements below are natural.

\setcounter{reqbase}{0}
\renewcommand{\thereqbase}{PU\arabic{reqbase}}
\begin{reqbase}\label{reqP:varphi}
	$\{\varphi_k\}$ is a positive, nondecreasing scalar sequence
	that either is eventually constant or grows without bound.
\end{reqbase}

\begin{reqbase}\label{reqP:boundedz}
	 If $\{\bz^k\}$ is unbounded, then $\varphi_k\to\infty$.
\end{reqbase}

\begin{reqbase}\label{reqP:large_enough_penalty}
	If $\varphi_k$ is eventually constant and equal to $\hat\varphi$, 
	and $\max\{\|\bbS^k\dvk\|, \\ \|\bbH\bbx^k +\bbc_{\varphi_k}-\bbA^T \dvk\|,|(\bH\bx^k +\bc-\bA^T\dvik)^T\bx^k|\}\to 0$,
	then {$\hat\varphi> {\rm lim\,inf}\|\dvik\|_\infty$}.
\end{reqbase}
While the above requirements allow for $\varphi_k$ to
be increased freely, the last one, stated next, insures that, when the 
stated assumptions are satisfied, $\varphi_k$ will eventually remain
constant indeed.  This is achieved by preventing situations
where
$\varphi_k$ is increased prematurely, based only
on \cref{reqP:large_enough_penalty}, with each increase of
$\varphi_k$ possibly triggering an initial increase 
of {$\|\dvik\|_\infty$},
in turn triggering a further increase of $\varphi_k$, resulting
in a runaway phenomenon. To this effect, it is important to give
a ``chance'' to the base iteration to recover from the disruption 
caused by an increase of $\varphi_k$, so $\|\dvik\|_\infty$ can
settle to a reasonably low value; i.e., not to rush to increase
it merely because it is again less than $\|\dvik\|_\infty$.   
Accordingly (since, for constant $\varphi$, \cref{reqB:cvgce}
implies convergence to a solution of \eqref{eq:primal_cqp_rho}),
the requirement below allows $\varphi_k$ to ``track'' $\|\dvik\|_\infty$
{\em only if} the iteration does not diverge away from optimality, as
indicated by growing duality measure or growing dual infeasibility.
Indeed, as it turns out, in addition to $\varphi_k$ not being already
much larger than $\|\dvik\|_\infty$, boundedness of distance to optimality,
together with boundedness of a certain inner product with $\bx^k$,
is sufficient.

\begin{reqbase}\label{reqP:phi_increase}
	If $\{\bz^k\}$ is bounded but $\varphi_k\to\infty$, 
	then there exists an infinite index set $K$ such that the following
	quantities are bounded on~$K$:
	\begin{subequations}\label{eq:E}
		\begin{align}\label{eq:E1}
		&\bbS^{k}\dvk
		\quad ({\rm i.e.,~} \bS^k\dvik {\rm ~and~}\bZ^k\dviik);
\\
		\label{eq:E2}
		&{\bbH}\bbx^{k} +\bbc_{\varphi_k}-{\bbA}^T \dvk
		\quad ({\rm i.e.,~} \bH\bx^k +\bc-\bA^T \dvik {\rm ~and~}\dvik+\dviik-\varphi_k\bone);
\\
		\label{eq:E3}
		&(\bH\bx^k +\bc-\bA^T \dvik)^T\bx^k
\\
		\label{eq:E4}
		&\sfrac{\varphi_k}{\max\{1,\|\dvik\|\}}\:.
		\end{align}	
	\end{subequations}
\end{reqbase}

\noindent
An instance of a penalty-parameter update that satisfies a more general
version (where equality constraints are allowed) of the Requirements~PU is given in section~\ref{subsec:penalty update}.

\section{Convergence Analysis for the Framework}
\label{sect:analysis}

Like the previous section, this section focuses exclusively
on the case of problems without equality constraints, i.e., $p=0$.
The general case is dealt with in section~\ref{sec:eq-constrained}.
The analysis in this section is strongly inspired from that 
in~\cite{He-Thesis} (and indirectly that in~\cite{HT12}),
in particular Lemmas~3.2 to~3.4 of~\cite{He-Thesis}, streamlined
and generalized here by allowing for the classes of base iterations 
and penalty-parameter updating rules specified in the previous section,
rather than being tailored to a specific base iteration and 
penalty-parameter updating rule.
It invokes the dual of~\eqref{eq:primal_cqp}, which is, when $p=0$,
\begin{equation}
\label{eq:dual_cqp}
\tag{D}
\maximize_{\bx\in\bbR^{n},\bspi\in\bbR^{m}}
\bspsi(\bx,\bspi):=-\frac{1}{2}\bx^T \bH\bx + \bspi^T\bb\,
{\rm ~s.t.~~}
\bH\bx+\bc-\bA^T\bspi = \bzero,
~\bspi\geq\bzero\:.
\end{equation}

Of course, a key for
the penalty-parameter updating approach to succeed is that $\varphi$ be (eventually) large
enough.

\begin{lemma} 
\label{lem:phi>lambda}
Suppose $(\bx^*,\bz^*,\bspi^*,\bsxi^*)$
solves~\textup{\eqref{eq:primal_cqp_rho}--\eqref{eq:dual_cqp_rho}} for
some $\varphi>\|\bspi^*\|_\infty$.  Then $\bz^*=\bzero$ and
$(\bx^*,\bspi^*)$ solves~\eqref{eq:primal_cqp}--\eqref{eq:dual_cqp}.
\end{lemma}
\begin{proof}
Since $\varphi>\|\bspi^*\|_\infty$, feasibility for \eqref{eq:dual_cqp_rho} implies that  $\dvii^*=\varphi\bone-\dvi^*>\bzero$.
Complementary slackness ($\bZ^*\bsxi^*=\bzero$) then implies that $\bz^*=\bzero$.
Therefore $(\bx^*,\bspi^*)$ is feasible, thus optimal, for~\eqref{eq:primal_cqp}--\eqref{eq:dual_cqp}.
\end{proof}

\begin{proposition}
\label{prop:constantPhi}
Suppose $\varphi_k$ is eventually constant.
Let $\hat\varphi:=\lim_{k\to\infty}\varphi_k$.
Then 
(i) the optimal solution set of \textup{(P$_{\hat\varphi}$)} is nonempty and bounded,
and (ii) as $k\to\infty$, $\bz^k\to\bzero$ and
$\bx^k$ converges to the optimal solution set of~\eqref{eq:primal_cqp}.
Furthermore, possible additional convergence properties 
(beyond \cref{reqB:cvgce}) of the specific base iteration under
consideration (with a feasible initial point) are preserved when
the initial point is infeasible for \eqref{eq:primal_cqp}.

\end{proposition}
\begin{proof}
Since $\varphi_k$ is eventually constant, \cref{reqP:boundedz} 
implies that $\{\bz^k\}$ is bounded.
From the third claim in \cref{lem:Pvarphi} and the facts that (i) $\{\bz^k\}$
is bounded, (ii) $\{(\bx^k,\bz^k)\}$ is feasible for~(P$_{\hat\varphi}$)
(\cref{reqB:descent}), and (iii) $\bff(\bx^k)+\hat\varphi\bone^T\bz^k$ 
monotonically decreases (\cref{rem:monotone}), it follows that
$\{\bbx^k\}$ is bounded.  
\cref{reqB:cvgce} then gives that 
$\max\{\|\bbS^k\dvk\|,\|\bbH\bbx^k +\bbc_{\varphi_k}-\bbA^T\dvk\|\}\to 0$,
which implies that $\bbx^k$ converges to the optimal solution set 
of~(P$_{\hat\varphi}$), and hence that (P$_{\hat\varphi}$) is bounded. 
Next, from boundedness of $\{\bbx^k\}$, we have (again invoking 
\cref{reqB:cvgce})
$|(\bH\bx^{k} +\bc-\bA^T\dvik)^T\bx^{k}|\to 0$.
\cref{reqP:large_enough_penalty} then leads to 
$\hat\varphi>\|\bspi^k\|_\infty$ for $k$ large enough.
It follows from \cref{lem:phi>lambda} and \cref{as:PD} that
the optimal solution set of~(P$_{\hat\varphi}$) is bounded.
Finally, from \cref{lem:phi>lambda}, $\bz^*=\bzero$ and 
$(\bx^k,\bspi^k)$ converges to the set of
primal--dual solutions to~\eqref{eq:primal_cqp}--\eqref{eq:dual_cqp}.
Also, because the key properties of~\eqref{eq:primal_cqp} 
(as listed in~\cref{as:PD}) are shared by~(P$_{\hat\varphi}$), 
all specific additional convergence properties of the base 
iteration are preserved.
\end{proof}

The next lemma
gives an upper bound on the magnitude of the relaxation variable~$\bz$
when $(\bx,\bz)$ is feasible for~\eqref{eq:primal_cqp_rho} and
$\varphi$ is large enough.
This upper bound is then used to prove boundedness of $\{\bz^k\}$ 
in \cref{lem:x_z_bounded}.
For use in the proofs here and in section~\ref{sec:eq-constrained}, recall 
that, because $\bH\succeq\bzero$,
\begin{equation}
\label{eq:positivesquare}
\hat\bx^T\bH\hat\bx+\bx^T\bH\bx-2(\hat\bx^T\bH\bx)
=(\hat\bx-\bx)^T\bH(\hat\bx-\bx)\geq0\:.
\end{equation}

\begin{lemma}
	\label{lem:technical}
	Let $(\hat\bx,\hat\bspi)$
	be feasible for~\eqref{eq:dual_cqp}
	and $(\bx,\bz)$ be feasible for~\eqref{eq:primal_cqp_rho}, and 
	let	$\varphi>\|\hat\bspi\|_\infty$. 
	Then
	\begin{equation}
	\label{eq:||z||<etc}
	\|\bz\|_\infty \leq 
	\frac{ \bff(\bx) + \varphi\bone^T\bz -\bspsi(\hat\bx,\hat\bspi) }
	{\varphi - \|\hat\bspi\|_\infty}\: .
	\end{equation}
\end{lemma}

\begin{proof}
	Feasibility of $(\bx,\bz)$ for~\eqref{eq:primal_cqp_rho} implies
	that $\bA\bx+\bz\geq\bb$ so that,
	since $\hat\bspi\geq\bzero$ (feasible for~\eqref{eq:dual_cqp}),
	\begin{equation}
	\hat\bspi^T \bA \bx + \hat\bspi^T\bz \geq \bb^T\hat\bspi.
	\end{equation}
	Since feasibility of $(\hat\bx,\hat\bspi)$ for~\eqref{eq:dual_cqp}
	implies $\bH\hat\bx + \bc = \bA^T\hat\bspi$, it follows that
	\begin{equation}
	\label{eq:2.6}
	-\hat\bspi^T\bz \leq (\bH\hat\bx+\bc)^T\bx - \bb^T\hat\bspi
	\leq \bff(\bx) -\bspsi(\hat\bx,\hat\bspi)\:,
	\end{equation}
	where we have used~\eqref{eq:positivesquare}.
	Since $\varphi>\|\hat\bspi\|_\infty$, $\hat{\bsxi}:=\varphi\bone - \hat\bspi>\bzero$.
	Adding $\varphi\bone^T\bz$ to both sides of~\eqref{eq:2.6} then yields
	\begin{equation}
	\hat\bsxi^T\bz 
	\leq \bff(\bx)
	+ \varphi\bone^T\bz -\bspsi(\hat\bx,\hat\bspi).
	\end{equation}
	Then, since $\bz\geq\bzero$ (feasible for~\eqref{eq:primal_cqp_rho}),
	\begin{equation}
	\hat\xi_i z_i 
	\leq \hat\bsxi^T\bz
	\leq \bff(\bx) 
	+ \varphi\bone^T\bz -\bspsi(\hat\bx,\hat\bspi),
	\end{equation}
	yielding, for $i=1,\ldots,m$,
	\begin{equation}
	z_i\leq 
	\frac{\bff(\bx) + \varphi\bone^T\bz -\bspsi(\hat\bx,\hat\bspi)}
	{\hat\xi_i}
	\leq \frac{\bff(\bx) + \varphi\bone^T\bz-\bspsi(\hat\bx,\hat\bspi)}
	{\varphi-\|\hat\bspi\|_\infty}.
	\end{equation}
	Since $\bz\geq\bzero$, the claim follows.
\end{proof}

\begin{lemma}
\label{lem:x_z_bounded}
Sequence $\{\bz^k\}$ is bounded.
\end{lemma}

\begin{proof}
Proceeding by contradiction, suppose $\{\bz^k\}$
is unbounded, so that, from \cref{reqP:boundedz},
$\varphi_k\to\infty$ as $k\to\infty$.
Then, 
given any~\eqref{eq:dual_cqp}--feasible $(\hat\bx,\hat\bspi)$, 
there exists $k_1$ such that $\varphi_k>\|\hat\bspi\|_\infty$
for all $k\geq k_1$ and in view of \cref{lem:technical},
\begin{equation}
\|\bz^{k-1}\|_\infty \leq \nu_k:=
\frac{\bff(\bx^{k-1}) + \varphi_k\bone^T\bz^{k-1}-\bspsi(\hat\bx,\hat\bspi)}
{\varphi_k - \|\hat\bspi\|_\infty}, \quad k\geq k_1.
\end{equation}
To conclude, we show that $\{\nu_k\}$ is bounded,
specifically, that $\nu_{k+1}\leq\nu_k$ for all $k\geq k_1$, contradicting
unboundedness of $\{\bz^k\}$.
Since (see \cref{rem:monotone}) 
$\bff(\bx^k)+\varphi_k\bone^T\bz^k \leq \bff(\bx^{k-1})+\varphi_k\bone^T\bz^{k-1}$
for all $k\geq1$,
it suffices to show that, for all $k\geq k_1$,
\begin{equation}
(\nu_{k+1}=)~\frac{\bff(\bx^{k}) + \varphi_{k+1}\bone^T\bz^{k}-\bspsi(\hat\bx,\hat\bspi)}
{\varphi_{k+1} - \|\hat\bspi\|_\infty}
\leq \frac{\bff(\bx^{k}) + \varphi_k\bone^T\bz^{k}-\bspsi(\hat\bx,\hat\bspi)}
{\varphi_k - \|\hat\bspi\|_\infty}.
\end{equation}
To that effect, we show that, for all $k$, the function 
$g_k:\bbR\to\bbR$ defined by 
\begin{equation}
\label{eq:gk}
g_k(\varphi):= \frac{\bff(\bx^k) + \varphi\bone^T\bz^k-\bspsi(\hat\bx,\hat\bspi)}
{\varphi - \|\hat\bspi\|_\infty}
\end{equation}
has a nonpositive derivative when $\varphi > \|\hat\bspi\|_\infty$.
Indeed,
\begin{equation}
g_k^\prime(\varphi) = 
-\frac{\bff(\bx^k)+\|\hat\bspi\|_\infty\bone^T\bz^k-\bspsi(\hat\bx,\hat\bspi)}
{(\varphi-\|\hat\bspi\|_\infty)^2}
\end{equation}
and
\begin{equation}
\begin{alignedat}{2}
\bff(\bx^k)+\|\hat\bspi\|_\infty\bone^T\bz^k-\bspsi(\hat\bx,\hat\bspi)
&=\frac{1}{2} \hat\bx^T\bH\hat\bx - \bb^T\hat\bspi 
+ \bff(\bx^k)+\|\hat\bspi\|_\infty\bone^T\bz^k\\
&\geq -\hat\bspi^T\bz^k + \|\hat\bspi\|_\infty\bone^T\bz^k\geq0\:,
\end{alignedat}
\end{equation}
where we have used the facts that, 
given any \eqref{eq:primal_cqp_rho}--feasible $(\bx,\bz)$
(and since $(\hat\bx,\hat\bspi)$ is \eqref{eq:dual_cqp}--feasible),
recalling~\eqref{eq:positivesquare},
\begin{equation}
\label{eq:ineq1}
\bff(\bx) + \frac{1}{2} \hat\bx^T\bH\hat\bx
\geq \bff(\bx) -\frac{1}{2}\bx^T\bH\bx + \hat\bx^T\bH\bx
= (\bc+\bH\hat\bx)^T\bx = \hat\bspi^T \bA \bx\:,
\end{equation}
and that, since $\hat\bspi\geq\bzero$ and $\bA\bx+\bz\geq\bb$,
\begin{equation}
\label{eq:ineq2}
\hat\bspi^T \bA \bx - \hat\bspi^T \bb 
= -\hat\bspi^T(\bb-\bA\bx)
\geq -\hat\bspi^T \bz\:.
\end{equation}
Since $\varphi_{k+1}\geq\varphi_k$ (\cref{reqP:varphi}), 
the proof is complete.
\end{proof}

It remains to show that, under \cref{reqP:phi_increase}, 
$\varphi_k$ is eventually constant, so \cref{prop:constantPhi}
applies. This is done in the next two lemmas and 
\cref{thm:varphi_bounded}.

\begin{lemma} 
\label{lem:z->0onK}
Suppose $\varphi_k\to\infty$ as $k\to\infty$.   Then
there exists an infinite index set $K$ that satisfies the
properties listed in \cref{reqP:phi_increase}.
Further, given any such $K$,
(i) $\bz^k\to\bzero$ on $K$ and 
(ii) $\{\bx^k\}$ is bounded on $K$.
\end{lemma}

\begin{proof}
Since $\varphi_k\to\infty$ as $k\to\infty$, 
boundedness of $\{\bz^k\}$ (\cref{lem:x_z_bounded})
and \cref{reqP:phi_increase} guarantee
existence of an 
infinite index set $K$ such that 
\eqref{eq:E1}--\eqref{eq:E3}
are bounded on $K$.
Let $(\hat\bx,\hat\bspi)$ be 
\eqref{eq:primal_cqp}--\eqref{eq:dual_cqp}--feasible so that (i)
$\hat\bs:=\bA\hat\bx-\bb~\geq\bzero$ and since
$\bA\bx^k+\bz^k-\bb=\bs^k$ for all $k$,
\begin{equation}
\label{eq:P-D}
\bA(\hat\bx-\bx^k)-\bz^k -(\hat\bs-\bs^k)=\bzero \quad\forall k,
\end{equation}
and (ii)
$\bA^T\hat\bspi=\bH\hat\bx+\bc$
and 
$\hat\bspi\geq\bzero$. 
Next,
\eqref{eq:E2}--\eqref{eq:E3} imply that
$(\hat\bx-\bx^k)^T (\bH\bx^k+\bc-\bA^T \dvik)$ is bounded for 
$k\in K$,
and adding $(\hat\bx-\bx^k)^T(\bA^T\hat\bspi-H\hat\bx-\bc)=0$ to it yields that
\begin{equation}
\label{eq:P-D2}
(\hat\bx-\bx^k)^T\bA^T(\hat\bspi- \dvik) - 
(\hat\bx-\bx^k)^T \bH (\hat\bx-\bx^k)
~{\rm is~bounded~for~}k\in K.
\end{equation}
Now we first show that, for some $C$,
\begin{equation}
\label{eq:XX}
(\hat\bx-\bx^k)^T \bH (\hat\bx-\bx^k) + \hat\bspi^T\bs^k
+ (\varphi_k\bone-\hat\bspi)^T\bz^k
\leq C \quad\forall k\in K \:.
\end{equation}
From~\eqref{eq:P-D}--\eqref{eq:P-D2}, we have, for some $\{\beta_k\}$
bounded on $K$,
\begin{equation}\label{eq:XXX}
(\hat\bx-\bx^k)^T \bH (\hat\bx-\bx^k)
= (\hat\bx-\bx^k)^T \bA^T (\hat\bspi- \dvik)+\beta_k
= (\hat\bs-\bs^k + {\bz^k})^T (\hat\bspi- \dvik)+\beta_k \:.
\end{equation}
Reorganizing and adding $\varphi_k(\bz^k)^T\bone$ to both sides
yields, for $k\in K$,
\begin{equation}
\begin{alignedat}{2}
(\hat\bx-\bx^k)^T &\bH (\hat\bx-\bx^k) + (\bs^k)^T\hat\bspi
+ (\bz^k)^T(\varphi_k\bone-\hat\bspi)\\
&= \hat\bs^T\hat\bspi - (\hat\bs - \bs^k +\bz^k)^T \dvik 
+ \varphi_k(\bz^k)^T\bone +\beta_k\\
&= \hat\bs^T\hat\bspi - \hat\bs^T \dvik
+ (\bs^k)^T \dvik + (\bz^k)^T (\varphi_k\bone- \dvik)
+\beta_k.
\end{alignedat}
\end{equation}
{Here the second term is nonpositive, and \cref{reqP:phi_increase} \eqref{eq:E1}--\eqref{eq:E2} implies that the third and fourth terms are bounded on $K$. Thus, the boundedness of $\{\beta_k\}$ on $K$ yields \eqref{eq:XX}.}
Next, note that each of the three terms on the left-hand side
of~\eqref{eq:XX} is bounded from below, so that 
all three are bounded on $K$.
Indeed, the first and second terms are nonnegative since 
$\bH\succeq\bzero$,	$\hat\bspi\geq\bzero$, and $\bs^k\geq\bzero$;
and the third term is nonnegative for $k$ large	enough since 
$\bz^k\geq\bzero$ and $\varphi_k\to\infty$.
Since $\varphi_k\to\infty$, claim (i) follows from boundedness
of the third term in the left-hand side of~\eqref{eq:XX}.

With~\eqref{eq:XX} in hand, 
invoking strict dual feasibility (\cref{as:PD}),
assume without loss of
generality that $\hat\bspi$ has strictly positive components.  
Then boundedness on $K$
of the second term in~\eqref{eq:XX} 
implies boundedness of $\{\bs^k\}$ on $K$.
From boundedness on $K$ of $\{\bz^k\}$ and $\{\bs^k\}$ 
and the definition
of $\bs^k$, it follows that $\{\bA\bx^k\}$ is bounded on $K$.
Also, since $\bH=\bH^T\succeq\bzero$,
boundedness on $K$ of the first term in~\eqref{eq:XX}
implies boundedness of 
$\bH\bx^k$, again on $K$.  
Finally, boundedness on
$K$ of $\{\bA\bx^k\}$ and $\{\bH\bx^k\}$ together with the full-rank 
assumption on $[\bH;\bA]$ (\cref{as:LI}) proves claim (ii).
\end{proof}

\begin{lemma} 
\label{lem:bddDag}
Suppose $\varphi_k\to\infty$ as $k\to\infty$ 
and let $K$ be as 
in \cref{lem:z->0onK}, so that $\bz^k\to\bz^*=\bzero$ on $K$, $\{\bx^k\}$ 
is bounded on $K$, and $K$ has the properties guaranteed
by \cref{reqP:phi_increase}.
Then, given any limit point $\bx^*$ of $\{\bx^{k}\}$ on $K$, there
exists a nonzero $\overline\bspi^*\geq\bzero$, 
such that
\begin{equation}
\bA^T\overline\bspi^*=\bzero, 
\quad \bS^*\overline\bspi^*=\bzero, 
\end{equation}
where $\bs^*:=\bA\bx^*+\bz^*-\bb =\bA\bx^*-\bb$.
\end{lemma}

\begin{proof}

First, we have from \cref{reqP:phi_increase} \eqref{eq:E2} that
$\varphi_k\bone - (\dvik + \dviik)$ is bounded on $K$.
Letting
$\overline\dvi^k:=\frac{1}{\varphi_k}\dvik$
and $\overline\dvii^{k}:=\frac{1}{\varphi_k}\dviik$,
we conclude that $\overline\dvi^{k}+\overline\dvii^{k}\to\bone$
on $K$
and in view of \cref{reqP:phi_increase} \eqref{eq:E4},
$\overline\dvi^{k}$ is bounded away from $\bzero$ on $K$.
Further, since (see Master Algorithm)
$(\dvik,\dviik)=\bblambda^k$ has nonnegative 
components,
$\overline\dvi^k$ and $\overline\dvii^k$
are bounded on $K$, hence have limit points on $K$,
and every limit point $\overline\dvi^*$ of $\overline\dvi^{k}$ on $K$
satisfies $\overline\dvi^*\geq\bzero$ and $\overline\dvi^*\not=\bzero$.
Finally, since $\bz^*=\bzero$ and $\{\bx^{k}\}$ is bounded on $K$ 
(\cref{lem:z->0onK} (ii)), 
boundedness of \eqref{eq:E1}--\eqref{eq:E2} in~\cref{reqP:phi_increase}
yields, by dividing through by $\varphi_k$,
\begin{equation}
\bA^T\overline\bspi^* = \bzero,\quad  \bS^*\overline\bspi^*=\bzero \:.
\end{equation}
\end{proof}

\begin{theorem}
\label{thm:varphi_bounded}
(i) $\varphi_k$ is eventually constant and (ii) as $k\to\infty$,
$\bz^k\to\bzero$ and 
$\bx^k$ converges to the optimal solution set of~\eqref{eq:primal_cqp}.
\end{theorem}
\begin{proof}
To prove claim (i), proceeding by contradiction, 
suppose that $\varphi_k\to\infty$
and {let $K$ be as in ~\cref{lem:z->0onK}.}
Then in view of~\cref{lem:z->0onK},
$\bz^k\to\bz^*:=\bzero$ on $K$ and $\{\bx^k\}$ is bounded on $K$.
Let $\bx^*$ be a limit points of $\{\bx^k\}$ on $K$.
From \cref{lem:bddDag}, there exists $\overline\bspi^*\neq\bzero$,
with $\overline\bspi^*\geq\bzero$,
such that 
\begin{equation}
\label{eq:Api=0,Spi=0}
\bA^T\overline\bspi^*=\bzero \quand \bS^*\overline\bspi^*=\bzero,
\end{equation}
i.e., $\overline\pi_i^*=0$ for all $i$ such that $s^*_i>0$,
where $\bs^*:=\bA\bx^*-\bb$.
Next, let $\bA_{\rm act}$ be the submatrix of $\bA$ associated with
active constraints at $\bx^*$ (i.e., the rows of $\bA_{\rm act}$ are all those rows of $\bA$ with
index $i$ such that $s_i^*=0$); and let $\overline\bspi_{\rm act}^*$
be the corresponding subvector of $\overline\bspi^*$.
Then
$\bA_{\rm act}\bx^* = \bb_{\rm act}$ and~\eqref{eq:Api=0,Spi=0}
imply that
\begin{equation}
\label{eq:Api=0}
\bA_{\rm act}^T\overline\bspi_{\rm act}^*=\bzero.
\end{equation}
Now, invoking~\cref{as:PD}, 
let $\hat\bx$ be strictly feasible 
for \eqref{eq:primal_cqp}, i.e.,
$\bA\hat\bx>\bb$, in particular, $\bA_{\rm act}\hat\bx>\bb_{\rm act}$.
With $\bv:=\hat\bx-\bx^*$, by subtraction, we get $\bA_{\rm act}\bv>\bzero$.
Left-multiplying both sides of~\eqref{eq:Api=0} by
$\bv^T$ yields $(\bA_{\rm act}\bv)^T\overline\bspi_{\rm act}^*=0$,
a contradiction since~\eqref{eq:Api=0,Spi=0} 
together with $\overline\bspi^*\neq\bzero$ and
$\overline\bspi^*\geq\bzero$ implies that
$\overline\bspi_{\rm act}^*\neq\bzero$ and
$\overline\bspi_{\rm act}^*\geq\bzero$.  This proves the first claim.
The second claim follows from~\cref{prop:constantPhi}.
\end{proof}

\section{Problems with Equality Constraints}
\label{sec:eq-constrained}

A standard approach for handling linear equality constraints within
an inequality-constrained optimization framework is, after constructing
an initial point that satisfies the equality constraints, to simply carry 
out the inequality-constrained optimization on the affine space defined 
by the equality constraints, rather than on $\bbR^n$.  Search directions 
based on the inequality constraints are thus projected on that subspace.   
A drawback of such approach is that possible sparsity of the 
equality-constraint matrix is not inherited by the projection operator.
Further, unless special care is taken, the initial equality-feasible
point may be far removed from the region of interest, as its construction
does not take the objective function into account.
An alternative approach, proposed in~\cite{Mayne1976} in a 
nonlinear-programming (NLP)
context, deals with one side of the (possibly nonlinear) equality 
constraints (e.g., the side that is satisfied by the initial point) 
as an {\em inequality} constraint, and uses an exact (and smooth) 
$\ell_1$ penalty function to drive the iterates to feasibility.  
A refined version
of this approach was later used in~\cite{TWBUL-2003} in an interior-point
NLP context.  

Inspired by the latter, we now formulate each scalar 
(linear) equality as {\em two} inequality constraints, i.e, 
we equivalently express \eqref{eq:primal_cqp} as
\begin{equation}
\label{eq:primal_cqp_Eq}
\tag{$\tilde{\textup{P}}$}
\minimize_{\bx\in\bbR^n} \:
\bff(\bx):=\frac{1}{2}\bx^T\bH\bx+\bc^T\bx \mbox{~s.t.~} \:
\bA\bx \geq \bb,\: \bC\bx\geq\bd,\: -\bC\bx\geq-\bd,
\end{equation}
which, as we will demonstrate, can be handled 
within the same infeasible-start framework.
Its dual is given by
\begin{equation}
\label{eq:dual_cqp_Eq}
\tag{$\tilde{\textup{D}}$}
\maximize_{\bx,\bspi,\bseta,\bszeta}
\bspsi(\bx,\bspi,\bseta,\bszeta)  
{\rm ~s.t.~~}
\bH\bx+\bc-\bA^T\bspi-\bC^T(\bseta-\bszeta) = \bzero,
(\bspi,\bseta,\bszeta)\geq\bzero,
\end{equation}
where $\bx\in\bbR^{n}$, $\bspi\in\bbR^{m}$ , $\bseta\in\bbR^{p}$, $\bszeta\in\bbR^{p}$, and 
\begin{equation}
\bspsi(\bx,\bspi,\bseta,\bszeta) :=
-\frac{1}{2}\bx^T \bH\bx + \bspi^T\bb + (\bseta-\bszeta)^T\bd \:.
\end{equation}
The corresponding augmented problem is
\begin{align}
\tag{$\tilde{\textup{P}}_{\varphi}$}
\label{eq:primal_cqp_phi_rho}
&\minimize_{\bx,\bz,\by} \:
\bff(\bx)+\varphi\bone^T[\bz;\by] \mbox{~s.t.~}
\bA\bx +\bz \geq \bb,\: \bz\geq\bzero,\: \bC\bx+\by\geq\bd,\: \bC\bx-\by\leq\bd\\
\label{eq:dual_cqp_phi_rho}
\tag{$\tilde{\textup{D}}_{\varphi}$}
&\maximize_{\bx,\bspi,\bsxi,\bseta,\bszeta}
\bspsi(\bx,\bspi,\bseta,\bszeta):=
-\frac{1}{2}\bx^T \bH\bx + \bspi^T\bb + \bseta^T\bd - \bszeta^T\bd\\
{\rm s.t.~}
&\bH\bx+\bc-\bA^T\bspi-\bC^T(\bseta-\bszeta) = \bzero,
\bspi+\bsxi=\varphi\bone, \bseta+\bszeta=\varphi\bone,
(\bspi,\bsxi,\bseta,\bszeta)\geq\bzero\nonumber
\end{align}
with $\bx\in\bbR^{n}$, $\bz\in\bbR^{m}$, $\by\in\bbR^{p}$, $\bspi\in\bbR^{m}$, $\bsxi\in\bbR^{m}$ $\bseta\in\bbR^{p}$, $\bszeta\in\bbR^{p}$.
We will also make use of the slack variables
\begin{equation}
\bt_+^k:=\bC\bx^k+\by^k-\bd\geq\bzero,~
\bt_-^k:=-\bC\bx^k+\by^k+\bd\geq\bzero;
\end{equation}
note that $\bt_+^k+\bt_-^k=2\by^k$.

\begin{remark}
Note the dissymmetry between the way original inequalities are augmented 
and the way inequalities issued from equalities are augmented in
\eqref{eq:primal_cqp_phi_rho}: unlike $\bz\geq\bzero$, $\by\geq\bzero$ is not
included.  While including $\by\geq\bzero$ would have simplified 
(by exploiting the 
symmetry) the expression of requirements for the penalty-parameter update as 
well as the ensuing analysis, the three sets of constraints involving 
$\by$ would then form a structurally linearly dependent set
(the difference of the first two is twice the third one), and because 
all three are active when $\by=\bzero$ (which is the case at the solution
when $\varphi$ is large enough) this may rule out some possible base
iterations (such as, in theory, CR-MPC).
\end{remark}

Substituting $[\bA;\bC;-\bC]$ for $\bA$, $(\bz,\by)$ for $\bz$ 
and $(\bspi,\bseta-\bszeta)$ for $\bspi$, we obtain the following 
revised list of requirements for the penalty-parameter updating rule.

\setcounter{reqbase}{0}
\renewcommand{\thereqbase}{PU\arabic{reqbase}$'$}

\begin{reqbase}\label{reqP:varphi_Eq}
        $\{\varphi_k\}$ is a positive, nondecreasing scalar sequence
        that either is eventually constant or grows without bound.
\end{reqbase}

\begin{reqbase}\label{reqP:boundedz_Eq}
If $\{(\bz^k,\by^k)\}$ is unbounded, then $\varphi_k\to\infty$.
\end{reqbase}

\begin{reqbase}\label{reqP:large_enough_penalty_Eq}
If $\varphi_k$ is eventually constant and equal to $\hat\varphi$, and 
$\|G_1^k\|$, $\|G_2^k\|$, and $|G_3^k|$ tend to zero, where
\begin{subequations}\label{eq:E_Eq}
\begin{align}
\label{eq:E1_Eq}
&G_1^k:=(\bS^k\bspi^k,~\bZ^k\bsxi^k,~\bT_+^k\bseta^k, ~\bT_-^k\bszeta^k), \\
\label{eq:E2_Eq}
& G_2^k:=\big(\bH\bx^k+\bc-\bA^T\bspi^k-\bC^T(\bseta^k-\bszeta^k),
\bspi^k+\bsxi^k-\varphi_k\bone, \bseta^k+\bszeta^k-\varphi_k\bone\big), \\
\label{eq:E3_Eq}
& G_3^k:=\big(\bH\bx^k+\bc-\bA^T\bspi^k-\bC^T(\bseta^k-\bszeta^k)\big)^T\bx^k,
\end{align}
\end{subequations}
then $\hat\varphi> {\rm lim\,inf}\|[\dvik;{\bseta^k-\bszeta^k}]\|_\infty$.
\end{reqbase}

\begin{reqbase}\label{reqP:phi_increase_Eq}
If $\varphi_k\to\infty$ and $\{(\bz^k,\by^k)\}$ is bounded, 
then there exists an infinite index set $K$ such that 
$G_1^k$, $G_2^k$, $G_3^k$,
and $\sfrac{\varphi_k}{\max\{1,\|[\dvik;{\bseta^k-\bszeta^k}]\|\}}$
are bounded on $K$.
\end{reqbase}

With \Crefrange{reqP:varphi_Eq}{reqP:phi_increase_Eq} substituted
for \Crefrange{reqP:varphi}{reqP:phi_increase} and
the Master Algorithm extended in the obvious way to account for 
the additional variable $\by^k$, with some adjustments, the convergence 
analysis in section~\ref{sect:analysis} extends to cases when equality constraints
are present, as we show next.
The following (readily proved) extended version of \cref{lem:Pvarphi}
will be used.

\renewcommand{\thecuslemma}{4.0}
\begin{cuslemma}
\label{lem:Pvarphi_Eq}
Given $\varphi>0$, \eqref{eq:primal_cqp_phi_rho} is strictly feasible.  Further, for
$\varphi>0$ large enough, \eqref{eq:primal_cqp_phi_rho} is bounded, i.e., has a
nonempty solution set.
Finally, for $\varphi>0$, given any $\rho>0$ and $\alpha\in\bbR$, the set
$\cS:=\{(\bx,\bz,\by)\in\cF_\alpha: \|[\bz;\by]\|\leq\rho\}$ is bounded,
where
$\cF_\alpha:=\{(\bx,\bz,\by):\bz\geq\bzero,~\by\geq\bzero,
~\bA\bx+\bz\geq\bb,~\bC\bx+\by\geq\bd,~-\bC\bx+\by\geq-\bd, 
~\bff(\bx)+\varphi\bone^T[\bz;\by]\leq\alpha
\}$.
\end{cuslemma}

As in section~\ref{sect:analysis}, we first show that, for sufficiently large penalty 
parameter $\varphi$, the solutions to the augmented primal--dual pair 
agree with the ones to the original primal--dual pair.
\begin{lemma} 
	\label{lem:phi>lambda_Eq}
	Suppose $(\bx^*,\bz^*,\by^*,\bspi^*,\bsxi^*,\bseta^*,\bszeta^*)$
	solves~\eqref{eq:primal_cqp_phi_rho}--\eqref{eq:dual_cqp_phi_rho} for
	some\\ ${\varphi>\|[\bspi^*;{\bseta^*-\bszeta^*}]\|_\infty}$.  
	Then $\bz^*=\bzero$, $\by^*=\bzero$, and
	$(\bx^*,\bspi^*,\bseta^*,\bszeta^*)$ 
	solves~\eqref{eq:primal_cqp_Eq}--\eqref{eq:dual_cqp_Eq}.
\end{lemma}
\begin{proof}
First, $\bz^*=\bzero$ follows as in the proof 
of \cref{lem:phi>lambda}.  Next, $\varphi>\|\bseta^*-\bszeta^*\|_\infty$ implies that 
$\varphi\bone>\bseta^*-\bszeta^*$, and since $\bseta^*,\bszeta^*\geq\bzero$,
$\varphi\bone>\bseta^*$; similarly, $\varphi\bone>\bszeta^*$.  Since feasibility
for \cref{eq:dual_cqp_phi_rho} implies $\bseta^*+\bszeta^*=\varphi\bone$,
it follows that $\bseta^*,\bszeta^*>\bzero$.  Complementary slackness then
implies that $\bC\bx^*-\bd+\by^*=\bzero$ and $\bC\bx^*-\bd-\by^*=\bzero$,
hence that $\by^*=\bzero$.  Therefore, $(\bx^*,\bspi^*,\bseta^*,\bszeta^*)$
is feasible, thus optimal, for~\eqref{eq:primal_cqp_Eq}--\eqref{eq:dual_cqp_Eq}.
\end{proof}

\begin{proposition}
	\label{prop:constantPhi_Eq}
	Suppose $\varphi_k$ is eventually constant.
	Then, $[\bz^k;\by^k]\to\bzero$ as $k\to\infty$ and
	$\bx^k$ converges to the optimal solution set of~\eqref{eq:primal_cqp_Eq}.
\end{proposition}

\begin{proof}
	The proof is identical to that of~\cref{prop:constantPhi}, subject 
	to replacing throughout $\bz^k$ with $[\bz^k;\by^k;\by^k]$, 
        $\bbx^k$ with $(\bx^k,\bz^k,\by^k)$, and $\bblambda^k$ 
	with $(\bspi^k,\bsxi^k,\bseta^k,\bszeta^k)$ and invoking 
        \cref{reqP:boundedz_Eq,reqP:large_enough_penalty_Eq}
        instead of
        \ref{reqP:boundedz} and \ref{reqP:large_enough_penalty},
        \cref{lem:phi>lambda_Eq} instead of \cref{lem:phi>lambda},
        and \cref{lem:Pvarphi_Eq} instead of \cref{lem:Pvarphi}.
\end{proof}

Similar to \cref{lem:technical}, the following lemma provides 
an upper bound on the magnitude of relaxation variables $\bz$ and $\by$
when $(\bx,\bz,\by)$ is feasible for~\eqref{eq:primal_cqp_phi_rho}
and $\varphi$ is large enough.
This bound is then used in \cref{lem:x_z_bounded_Eq} 
to show boundedness of $\{(\bz^k,\by^k)\}$.
Note that, in contrast with \cref{lem:technical}, \cref{lem:technical_Eq} assumes a lower bound on the penalty parameter $\varphi$ that is more restrictive than the one in \cref{lem:phi>lambda_Eq}.
This however does not interfere with the analysis, since \cref{lem:technical_Eq} is only invoked in the proof of \cref{lem:x_z_bounded_Eq}, in which $\varphi$ is assumed (in a contradiction argument) to be unbounded.

\begin{lemma}
\label{lem:technical_Eq}
Let $(\hat\bx,\hat\bspi, { \hat\dvei, \hat\dveii})$
be feasible for~\eqref{eq:dual_cqp_Eq}
and $(\bx,\bz {, \by})$ be feasible for \eqref{eq:primal_cqp_phi_rho}, and 
let $\varphi>{\|[\hat\bspi;2\hat\bseta;2\hat\bszeta]\|_\infty}$.
Then
\begin{equation}
\|[\bz;\by]\|_\infty \leq 
\frac{\bff(\bx)+\varphi\bone^T[\bz;\by]-\bspsi(\hat\bx,\hat\bspi,\hat\bseta,\hat\bszeta)}
{{\half}\big(\varphi - \|[\hat\bspi;{2\hat\bseta;2\hat\bszeta}]\|_\infty\big)} \: .
\end{equation}
\end{lemma}

\begin{proof}
Feasibility of $(\bx,\bz,\by)$ for \eqref{eq:primal_cqp_phi_rho} implies
that $\bA\bx+\bz\geq\bb$, $\bC\bx+\by\geq\bd$, and $-\bC\bx+\by\geq-\bd$ so that,
since $\hat\dvi\geq\bzero$, $\hat\dvei\geq\bzero$, and $\hat\dveii\geq\bzero$ 
(feasible for~\eqref{eq:dual_cqp_Eq}),
\begin{equation}
\hat\bspi^T \bA \bx + \hat\bspi^T\bz \geq \bb^T\hat\bspi, 
\quad\hat\bseta^T \bC \bx + \hat\bseta^T\by \geq \hat\bseta^T\bd,
\quad  -\hat\bszeta^T \bC \bx + \hat\bszeta^T\by \geq -\hat\bszeta^T\bd\:.
\end{equation}
Since feasibility of $(\hat\bx,\hat\bspi,\hat\bseta,\hat\bszeta)$ 
for~\eqref{eq:dual_cqp_Eq}
implies $\bH\hat\bx + \bc = \bA^T\hat\bspi+\bC^T(\hat\bseta-\hat\bszeta)$, 
it follows that
\begin{equation}
\label{eq:2.6_Eq}
-\hat\bspi^T\bz - \hat\bseta^T \by  - \hat\bszeta^T\by
\leq (\bH\hat\bx+\bc)^T\bx - \bb^T\hat\bspi 
-\hat\dvei^T \bd + \hat\dveii^T \bd
\leq \bff(\bx) -\bspsi(\hat\bx,\hat\bspi,\hat\bseta,\hat\bszeta),
\end{equation}
where we again used~\eqref{eq:positivesquare}.
Since $\varphi>\|[\hat\bspi;{2\hat\bseta;2\hat\bszeta}]\|_\infty$, 
we have
$\hat{\bsxi}:=\varphi\bone - \hat\bspi>\bzero$,
$\hat\bsalpha_1:={\half}\varphi\bone-\hat\bseta>\bzero$,
and $\hat\bsalpha_2:={\half}\varphi\bone-\hat\bszeta>\bzero$.
Adding $\varphi(\bone^T\bz +\bone^T\by)$ to both sides 
of~\eqref{eq:2.6_Eq} then yields
\begin{equation}
[\hat\bsxi;\hat\bsalpha_1;\hat\bsalpha_2]^T[\bz;\by;\by]
\leq \bff(\bx)
+ \varphi\bone^T[\bz;\by] -\bspsi(\hat\bx,\hat\bspi,\hat\bseta,\hat\bszeta).
\end{equation}
Then, since 
$(\bz,\by)\geq\bzero$
(feasible for~\eqref{eq:primal_cqp_phi_rho}) 
\begin{equation}
[\hat\bsxi;\hat\bsalpha_1;\hat\bsalpha_2]_i[\bz;\by;\by]_i
\leq [\hat\bsxi;\hat\bsalpha_1;\hat\bsalpha_2]^T[\bz;\by;\by]
\leq \bff(\bx) 
+  \varphi\bone^T[\bz;\by] -\bspsi(\hat\bx,\hat\bspi,\hat\bseta,\hat\bszeta),
\end{equation}
yielding, for $i=1,\ldots,m+2p$,
\begin{equation}
[\bz;\by;\by]_i\leq 
\frac{\bff(\bx) + \varphi\bone^T[\bz;\by] - \bspsi(\hat\bx,\hat\bspi,\hat\bseta,\hat\bszeta)}
{[\hat\bsxi;\hat\bsalpha_1;\hat\bsalpha_2]_i}
\leq \frac{\bff(\bx) + \varphi\bone^T[\bz;\by] - \bspsi(\hat\bx,\hat\bspi,\hat\bseta,\hat\bszeta)}
{{\half}\big(\varphi-\|[\hat\bspi;{2}\hat\bseta;{2}\hat\bszeta]\|_\infty\big)}\:,
\end{equation}
where the last inequality can be verified by noting (i) that
$(\hat\alpha_1)_j\geq \frac{1}{2}\varphi-\|\hat\bseta\|_\infty$
and $(\hat\alpha_2)_j\geq \frac{1}{2}\varphi-\|\hat\bszeta\|_\infty$,
(ii) that since $\varphi-\|\hat\bspi\|_\infty>0$,
$\hat\xi_j\geq\varphi-\|\hat\bspi\|_\infty
>\frac{1}{2}(\varphi-\|\hat\bspi\|_\infty)$,
and (iii) that this implies that
$[\hat\bsxi;\hat\bsalpha_1;\hat\bsalpha_2]_i
\geq\min\{\frac{1}{2}(\varphi-\|\hat\bspi\|_\infty), 
\frac{1}{2}\varphi-\|\hat\bseta\|_\infty,
\frac{1}{2}\varphi-\|\hat\bszeta\|_\infty\}$.
Since $(\bz,\by)\geq\bzero$, the claim follows.
\end{proof}

\begin{lemma}
\label{lem:x_z_bounded_Eq}
Sequence $\{(\bz^k,\by^k)\}$ is bounded.
\end{lemma}
\begin{proof}
The proof is identical to that of~\cref{lem:x_z_bounded}
(invoking \cref{lem:technical_Eq} instead of \cref{lem:technical}),
subject to replacing throughout $\bz^k$ with 
$(\bz^k,\by^k)$,
$\bspsi(\hat\bx,\hat\bspi)$
with $\bspsi(\hat\bx,\hat\bspi,\hat\bseta,\hat\bszeta)$, and 
$\|\hat\bspi\|_\infty$ with $\|[\hat\bspi;{2\hat\bseta;2\hat\bszeta}]\|_\infty$.
\end{proof}

\begin{lemma}
\label{lem:z->0onK_Eq}
Suppose $\varphi_k\to\infty$ as $k\to\infty$.   Then
there exists an infinite index set $K$ such that
(i) $(\bz^k,\by^k)\to\bzero$ on $K$ and 
(ii) $\{\bx^k\}$ is bounded on $K$.
\end{lemma}

\begin{proof}
The proof of~\cref{lem:z->0onK} is adapted as follows,
with \cref{reqP:phi_increase_Eq} now being invoked
instead of \cref{reqP:phi_increase}.
Let $(\hat\bx,\hat\bspi,\hat\bseta,\hat\bszeta)$ be
\eqref{eq:primal_cqp_Eq}--\eqref{eq:dual_cqp_Eq}--feasible,
$\hat\bs:=\bA\hat\bx-\bb(\geq\bzero)$, and 
$\hat\bt_+:=\hat\bt_-:=\bzero$.  
Then~\eqref{eq:P-D} becomes
a set of three equations:
\begin{subequations}\label{eq:z->0onK_Eq_1}
\begin{align}
\bA(\hat\bx-\bx^k)-\bz^k -(\hat\bs-\bs^k)&=\bzero \quad\forall k\\
\bC(\hat\bx-\bx^k)-\by^k - (\hat\bt_+-\bt_+^k)&=\bzero \quad\forall k\\
-\bC(\hat\bx-\bx^k)-\by^k - (\hat\bt_--\bt_-^k)&=\bzero \quad\forall k\:.
\end{align}
\end{subequations}
Dual feasibility combined with \eqref{eq:z->0onK_Eq_1} now yields
(replacing~\eqref{eq:P-D2}) boundedness on $K$ of
\begin{equation}
(\hat\bx-\bx^k)^T(
\bA^T(\hat\bspi- \dvik) + \bC^T((\hat\bseta-\bseta^k)-(\hat\bszeta-\bszeta^k)))
- (\hat\bx-\bx^k)^T \bH (\hat\bx-\bx^k),
\end{equation}
and it now follows that there exists $D>0$ such 
that, for all $k\in K$,
\begin{equation}
\label{eq:equal-XX}
(\hat\bx-\bx^k)^T \bH (\hat\bx-\bx^k) + \hat\bspi^T\bs^k + \hat\bseta^T\bt_+^k
+ \hat\bszeta^T\bt_-^k
+ (\varphi_k\bone-\hat\bspi)^T\bz^k
+ (\varphi_k\bone-(\hat\bseta+\hat\bszeta))^T\by^k
\leq D\:,
\end{equation}
replacing~\eqref{eq:XX}.
The proof concludes essentially like that of \cref{lem:z->0onK}.

The details that lead to~\eqref{eq:equal-XX} are as follows.
Equation \eqref{eq:XXX} becomes
\begin{equation}
\begin{alignedat}{2}
&(\hat\bx-\bx^k)^T \bH (\hat\bx-\bx^k)
= (\hat\bx-\bx^k)^T (\bA^T (\hat\bspi- \dvik) 
                     + \bC^T((\hat\bseta-\bseta^k)-(\hat\bszeta-\bszeta^k)))
+\beta_k\\
=&(\hat\bs-\bs^k + {\bz^k})^T (\hat\bspi- \dvik)
+(\hat\bt_+-\bt_+^k+\by^k)(\hat\bseta-\bseta^k)
+(\hat\bt_--\bt_-^k+\by^k)(\hat\bszeta-\bszeta^k)
+\beta_k \:,
\end{alignedat}
\end{equation}
where $\beta_k$ is bounded on $K$ and the second equality follows 
from \eqref{eq:z->0onK_Eq_1}.
Upon adding $\varphi_k\bone^T[\bz^k;\by^k]$ to both sides and
reorganizing, we get (since $\hat\bt_+=\hat\bt_-=\bzero$)
\begin{equation}
\begin{alignedat}{2}
&(\hat\bx-\bx^k)^T \bH (\hat\bx-\bx^k) + \hat\bspi^T\bs^k + \hat\bseta^T\bt_+^k
+ \hat\bszeta^T\bt_-^k\\
&\qquad\qquad\qquad+ (\varphi_k\bone-\hat\bspi)^T\bz^k
+ (\varphi_k\bone-(\hat\bseta+\hat\bszeta))^T\by^k\\
&=\hat\bs^T\hat\bspi - \hat\bs^T\bspi^k
+ (\bs^k)^T\bspi^k + (\bt_+^k)^T\bseta^k + (\bt_-^k)^T\bszeta^k \\
&\qquad\qquad\qquad+ (\varphi_k\bone-\bspi^k)^T\bz^k
+ (\varphi_k\bone-(\bseta^k+\bszeta^k))^T\by^k
+\beta_k \:,
\end{alignedat}
\end{equation}
and essentially the same analysis as is done in the proof of~\cref{lem:z->0onK}
applies here, concluding the proof.
\end{proof}

In the remainder of this section, the difference $\bseta-\bszeta$ plays
a key role, so we define $\bsomega:=\bseta-\bszeta$,
and similarly, $\bsomega^k:=\bseta^k-\bszeta^k$, etc.

\begin{lemma}
\label{lem:bddDag_Eq}
Suppose $\varphi_k\to\infty$ as $k\to\infty$ and let $K$ be as in
\cref{lem:z->0onK_Eq}, so that 
$(\bz^k,\by^k)\to(\bz^*,\by^*)=(\bzero,\bzero)$ on $K$, 
$\{\bx^k\}$ is bounded on $K$, and
$K$ has the properties guaranteed by \cref{reqP:phi_increase_Eq}.
Then, given any limit point $\bx^*$ of $\{\bx^k\}$ on $K$, 
there exist $\overline\bspi^*\in\bbR_+^m$ (if $m>0$)
and $\overline\bsomega^*\in\bbR^p$ (if $p>0$), 
with $(\overline\bspi^*,\overline\bsomega^*)\neq\bzero$,
such that 
\begin{equation}
\bA^T\overline\bspi^* + \bC^T\overline\bsomega^*=\bzero,
\quad \bS^*\overline\bspi^*=\bzero,
\end{equation}
where $\bs^*:=\bA\bx^*+\bz^*-\bb=\bA\bx^*-\bb$.
\end{lemma}

\begin{proof}
First, we have from \cref{reqP:phi_increase_Eq} ($G_2^k$) 
that $\varphi_k\bone-(\bspi^k+\bsxi^k)$ and
$\varphi_k\bone-(\bseta^k+\bszeta^k)$ are bounded on $K$, which
implies that $\overline\bspi^k+\overline\bsxi^k\to\bone$ on $K$
and $\overline\bseta^k+\overline\bszeta^k\to\bone$ on $K$, where, again
$\overline\bspi^k:=\bspi^k/\varphi_k$, $\overline\bsxi^k:=\bsxi^k/\varphi_k$,
$\overline\bseta^k:=\bseta^k/\varphi_k$, and $\overline\bszeta^k:=\bszeta^k/\varphi_k$.
Additionally, in view of \cref{reqP:phi_increase_Eq}, 
$(\overline\bspi^k,\overline\bsomega^k:=(\bseta^k-\bszeta^k)/\varphi_k)$ is bounded away from $\bzero$
on $K$, and
since $\overline\bspi^k$, $\overline\bsxi^k$, $\overline\bseta^k$,
and $\overline\bszeta^k$ all have nonnegative components, they
are all bounded on $K$, and so is 
$\overline\bsomega^k$.
Hence all have limit points on $K$, and
every limit points $(\overline\bspi^*,\overline\bsomega^*)$ of
$(\overline\bspi^k,\overline\bsomega^k)$ satisfies 
$\overline\bspi^*\geq\bzero$ and
$(\overline\bspi^*,\overline\bsomega^*)\neq\bzero$.
Finally, since $(\bz^*,\by^*)=(\bzero,\bzero)$ and
$\{\bx^{k}\}$ is bounded on $K$ (\cref{lem:z->0onK_Eq} (ii)), 
boundedness of $G_1^k$ and $G_2^k$ in \cref{reqP:phi_increase_Eq}
yields, by dividing through by $\varphi_k$,
\begin{equation}
\bA^T\overline\bspi^*+\bC^T\overline\bsomega^*=\bzero,\quad  \bS^*\overline\bspi^*=\bzero \:.
\end{equation}
\end{proof}

\begin{theorem}
\label{thm:varphi_bounded_Eq}
(i) $\varphi_k$ is eventually constant and (ii)
$(\bz^k,\by^k)\to\bzero$ as $k\to\infty$ and
$\bx^k$ converges to the set of solutions of~\eqref{eq:primal_cqp_Eq}.
\end{theorem}
\begin{proof}
Following the proof of~\cref{thm:varphi_bounded},
invoking~\cref{lem:bddDag_Eq}, we note that for 
some $\overline\bspi^*\geq\bzero$
(if $m>0$) and $\overline\bsomega^*$ (if $p>0$), 
with $(\overline\bspi^*,\overline\bsomega^*)\neq\bzero$,
we have
\begin{equation}
\label{eq:ap+cw_and_sPi}
\bA^T\overline\bspi^* + \bC^T\overline\bsomega^*=\bzero,
\quad\text{with}\quad
\overline\pi_i^*=0 \quad\forall i\in\{i\colon s_i^*>0\}\:,
\end{equation}
where $\bs^*:=\bA\bx^*-\bb$.
Next, let $\bA_{\rm act}$ be the submatrix of $\bA$ associated with 
active constraints at $\bx^*$;
i.e., the rows of $\bA_{\rm act}$ are all those rows of $\bA$ with
index $i$ such that $s_i^*=0$.
Then
$\bA_{\rm act}\bx^* = \bb_{\rm act}$, $\bC\bx^*=\bd$, 
and~\eqref{eq:ap+cw_and_sPi} imply that
\begin{equation}
\label{eq:Api+CTw=0}
\bA_{\rm act}^T\overline\bspi_{\rm act}^* + \bC^T\overline\bsomega^*=\bzero.
\end{equation}
Now, invoking \cref{as:PD}, let $\hat\bx$ be strictly 
feasible for \eqref{eq:primal_cqp}, i.e., 
$\bC\hat\bx=\bd$ and $\bA\hat\bx>\bb$,
in particular, $\bA_{\rm act}\hat\bx>\bb_{\rm act}$.  
Proceeding as in the proof of \cref{thm:varphi_bounded}, we conclude
that $\overline\bspi_{\rm act}^*=\bzero$ and, from~\eqref{eq:Api+CTw=0},
that $\bC^T\overline\bsomega^*=\bzero$,
a contradiction since $\bC$ has full row rank (\cref{as:LI}), proving
the first claim.  The second claim follows from~\cref{prop:constantPhi_Eq}.
\end{proof}

\section{Certificate of Infeasibility}
\label{sec:infeasibility}
The assumption (part of \cref{as:PD}) that \eqref{eq:primal_cqp} has a (strictly)
feasible point generally cannot be ascertained at the outset,
and in case of infeasibility it is desirable that the sequences 
generated by the algorithm provide, preferably early on, a {\em certificate
of infeasibility}.  
In this section it is shown that the
proposed framework does provide such certificate and that, in
addition, it provides an initial feasible point for a nearby
feasible problem.

Thus, in this section, \cref{as:PD} is replaced with the following
less restrictive assumption (primal feasibility is not assumed),
involving auxiliary problem 
\begin{equation}
\label{eq:primal_cqp_relaxed}
\minimize_{\bx\in\bbR^n} \:
\bff(\bx) \mbox{~s.t.~} \:
\bA\bx \geq \bb',\: \bC\bx\geq\bd'_-,\: \bC\bx\leq\bd'_+\:,
\tag{P$'$}
\end{equation}
a feasible relaxation of the infeasible \eqref{eq:primal_cqp}, 
with some $\bb'\leq\bb$, $\bd'_-\leq\bd$ and $\bd'_+\geq\bd$ 
selected in such a way that
\eqref{eq:primal_cqp_relaxed} is indeed feasible.

\renewcommand\theassumptionPrime{1$'$}

\begin{assumptionPrime}
\label{as:shifted_bounded_solution_set}
\eqref{eq:primal_cqp_relaxed} has a (nonempty) bounded optimal solution set.
\end{assumptionPrime}
Note that \cref{as:shifted_bounded_solution_set} implies feasibility 
of the dual of \eqref{eq:primal_cqp_relaxed}, which is equivalent 
to feasibility of~\eqref{eq:dual_cqp}.
Lemmas~\ref{lem:phi>lambda_Eq} 
and~\ref{lem:x_z_bounded_Eq}, invoked in the analysis below,
were established without using the primal-feasibility nor 
strict-dual-feasibility portions of
\cref{as:PD}, so that the less 
restrictive \cref{as:shifted_bounded_solution_set} is sufficient there.
The following additional assumption is also invoked.
\begin{assumption}\footnote{%
Numerical experimentation, including tests with small-size problems
that do not satisfy \cref{as:bounded_feasible_set},
suggests that mere boundedness of the {\em optimal} solution
set of \eqref{eq:primal_cqp_relaxed}, as implied by \cref{as:shifted_bounded_solution_set}, is sufficient for
the results to hold.  A proof of this is elusive at this time though.
In any case, boundedness of the feasible set of course can be achieved
by imposing appropriately large bounds to the components of $\bx$.
}
\label{as:bounded_feasible_set}
\eqref{eq:primal_cqp_relaxed} has a bounded feasible set.
\end{assumption}

The notation used below is as in section~\ref{sec:eq-constrained}.
In particular, $\bsomega:=\bseta-\bszeta$, $\bsomega^k:=\bseta^k-\bszeta^k$,
etc.
It is well known (Farkas's Lemma; 
e.g.~\cite[Proposition 6.4.3(iii)]{MatusekGartner07}) that a system of the
form $\bA\bx\geq\bb,~\bC\bx=\bd$ is infeasible, i.e., has no solution,
if and only if there exists $(\bspi,\bsomega)$ such that
\begin{equation}
\label{eq:infeasible_condition}
\bspi\geq\bzero, \quad \bA^T\bspi + \bC^T\bsomega=\bzero,
\quad \bb^T\bspi+\bd^T\bsomega>0\:.
\end{equation}
Now, consider the reparameterization/rescaling of \eqref{eq:primal_cqp_phi_rho} obtained by
defining $\alpha:=1/\varphi$ and scaling the objective function
by $\alpha$, viz. 
\begin{equation}
\label{eq:alpha-penalization}
\minimize_{\bx,\bz,\by}~\alpha \bff(\bx) + \bone^T[\bz;\by]
{\rm ~s.t.~} \bA\bx+\bz\geq\bb, \bC\bx+\by\geq\bd, -\bC\bx+\by\geq-\bd, 
\bz\geq\bzero
\tag{$\tilde{\textup{P}}_\alpha$}
\end{equation}
with $\bx\in\bbR^n$, $\bz\in\bbR^{m}$, $\by\in\bbR^{p}$.
The limit problem (with $\alpha=0$) is
\begin{equation}
\label{primalOfLimitProblem}
\minimize_{\bx,\bz,\by}~ \bone^T[\bz;\by] \quad{\rm~s.t.~} ~\bA\bx+\bz\geq\bb, 
~\bC\bx+\by\geq\bd, ~-\bC\bx+\by\geq-\bd,
~\bz\geq\bzero\:,
\tag{$\tilde{\textup{P}}_0$}
\end{equation}
with dual
\begin{equation}
\label{dualOfLimitProblem}
\maximize_{\bspi\in\bbR^{m},\bsomega\in\bbR^{p}}~\bb^T\bspi+\bd^T\bsomega                
\quad{\rm ~s.t.~} ~\bA^T\bspi+\bC^T\bsomega=\bzero,
~\bsomega\in[-1,1],
~\bspi\in[0,1]\:,
\tag{$\tilde{\textup{D}}_0$}
\end{equation}
and optimality conditions given by 
\begin{equation}
\begin{alignedat}{2}
&\bA^T\bspi+\bC^T(\bseta-\bszeta)=\bzero, 
~\bspi+\bsxi=\bone, ~\bseta+\bszeta=\bone,\\
~\bS\bspi=\bzero, ~&\bZ\bsxi=\bzero, ~\bT_+\bseta=\bzero, ~\bT_-\bszeta=\bzero,
~(\bs,\bt_+,\bt_-,\bz,\bspi,\bsxi,\bseta,\bszeta)\geq\bzero.
\end{alignedat}
\end{equation}

The analysis proceeds as follows.
\begin{lemma}\label{lem:infeasible_penalty}
If \eqref{eq:primal_cqp} is infeasible, then $\alpha_k\to 0$ as $k\to\infty$.
\end{lemma}
\begin{proof}
By contradiction.  If $\alpha_k$ does not tend to 0, then it is
eventually constant, say, equal to $\hat\alpha>0$. 
Thus, \cref{reqP:boundedz_Eq} implies that $\{(\bz^k,\by^k)\}$ is bounded.
Since ($\tilde{\textup{P}}_{\hat\alpha})$
is strictly feasible, \cref{as:shifted_bounded_solution_set} and \cref{lem:Pvarphi_Eq} imply 
boundedness of its constrained level sets,
hence boundedness of $\{\bx^k\}$, and \cref{reqB:cvgce,reqP:large_enough_penalty_Eq}
imply convergence to the
set of optimal solution and 
$\hat\varphi:=1/\hat\alpha>\liminf\|[\bspi^*;\bsomega^*]\|_\infty$.
\cref{lem:phi>lambda_Eq} 
then implies that $(\bx^*,\bspi^*,\bseta^*,\bszeta^*)$ 
solves \eqref{eq:primal_cqp}--\eqref{eq:dual_cqp}, in
contradiction with \eqref{eq:primal_cqp} being infeasible.
\end{proof}

\begin{lemma}
\label{lem:infeas}
If \eqref{eq:primal_cqp} is infeasible, then
there exists an infinite index set $K$ such that,
as $k\to\infty$, $k\in K$,
$(\overline\bspi^k,\overline\bsomega^k)$,
with $\overline\bspi^k:=\alpha_k\bspi^k$ and 
$\overline\bsomega^k:=\alpha_k\bsomega^k$,
tends to the solution set of
\eqref{dualOfLimitProblem}
and $(\bz^k,\by^k)$ tends to the solution set 
of \eqref{primalOfLimitProblem}.
\end{lemma}
\begin{proof}
Let $K:=\{k:\alpha_{k+1}<\alpha_k\}$, an infinite index set 
in view of \cref{lem:infeasible_penalty}.
Also, let $\overline\bsxi^k:=\alpha_k\bsxi^k$, 
$\overline\bseta^k:=\alpha_k\bseta^k$, and 
$\overline\bszeta^k:=\alpha_k\bszeta^k$.
In view of \cref{lem:x_z_bounded_Eq} 
(boundedness of $\{(\bz^k,\by^k)\}$) and \cref{reqP:phi_increase_Eq},
since $\alpha_k\to0$, we have, together with 
$(\overline\bspi^k,
\overline\bsxi^k,\overline\bseta^k,\overline\bszeta^k)\geq\bzero$,
\begin{equation}
\begin{alignedat}{2}
&\bS^k\overline\bspi^k\to\bzero,
~~\bZ^k\overline\bsxi^k\to\bzero,
~~\bT_+^k\overline\bseta^k\to\bzero, 
~~\bT_-^k\overline\bszeta^k\to\bzero,\\
\alpha_k(\bH\bx^k+\bc)&-\bA^T\overline\bspi^k-\bC^T\overline\bsomega^k\to\bzero, 
~~\overline\bspi^k+\overline\bsxi^k-\bone\to\bzero, 
~~\overline\bseta^k+\overline\bszeta^k-\bone\to\bzero, 
\end{alignedat}
\end{equation}
as $k\to\infty$, $k\in K$.
In particular, $(\overline\bspi^k,
\overline\bsxi^k,\overline\bseta^k,\overline\bszeta^k)$ is bounded on $K$.
In view of \cref{as:bounded_feasible_set}, $\{\bx^k\}$ is also bounded, 
and it follows that for any limit point 
$(\hat\bx,\hat\bz,\hat\by,\hat\bspi,\hat\bsxi,\hat\bseta,\hat\bszeta)$
of $(\bx^k,\bz^k,\by^k,
\overline\bspi^k,\overline\bsxi^k,\overline\bseta^k,\overline\bszeta^k)$ 
on $K$,
\begin{equation}
\hat\bS\hat\bspi=\bzero,~\hat\bT_+\hat\bseta=\bzero,~\hat\bT_-\hat\bszeta=\bzero,
~\hat\bZ\hat\bsxi=\bzero,~\bA^T\hat\bspi+\bC^T\hat\bsomega=\bzero,
~\hat\bspi+\hat\bsxi=\bone,~\hat\bseta+\hat\bszeta=\bone,
\end{equation}
implying the claim.
\end{proof}

\begin{lemma}
If \eqref{eq:primal_cqp} is infeasible, every solution 
$(\bspi^*,\bsomega^*)$
of \eqref{dualOfLimitProblem} 
satisfies $\bspi^*\geq\bzero$, $\bA^T\bspi^*+\bC^T\bsomega^*=\bzero$, 
and $\bb^T\bspi^*+\bd^T\bsomega^*>0$.
\end{lemma}
\begin{proof}
Immediate consequence of strong duality, since the dual of \eqref{primalOfLimitProblem} is
\eqref{dualOfLimitProblem}.
\end{proof}
Together, these three lemmas establish the following. 
\begin{theorem}\label{thm:infeasible_certificate}
If \eqref{eq:primal_cqp} is infeasible, then given $\epsilon>0$, there exists ${\hat{k}}$ such 
that
\begin{equation}
\|\bA^T\overline\bspi^{\hat{k}} + \bC^T\overline\bsomega^{\hat{k}}\|\leq\epsilon, 
\quad \bb^T\overline\bspi^{\hat{k}} + \bd^T\overline\bsomega^{\hat{k}}>0,
\end{equation}
and 
\begin{equation}
\bone^T[\bz^{\hat{k}};\by^{\hat{k}}] \leq \bone^T[\hat\bz;\hat\by] + \epsilon,
\end{equation}
where $\overline\bspi^{\hat{k}}\geq\bzero$ and $(\hat\bz,\hat\by)$ solves \eqref{primalOfLimitProblem}.
\end{theorem}
Hence, if \eqref{eq:primal_cqp} is infeasible, the Master Algorithm provides a
certificate of (approximate) infeasibility, 
as well as an $\epsilon$-$\ell_1$-least relaxation of the constraints,
replacing
$\bb$ with $\bb':=\bb-\bz^{\hat{k}}$, and ``spreading'' $\bC\bx=\bd$ to 
$-\Delta\bd_- \leq \bC\bx-\bd \leq \Delta\bd_+$, with
\[
(\Delta\bd_+)_i:=
\begin{cases}
(\by^{\hat{k}})_i\:,& \text{if } (\bC\bx^{\hat{k}} - \bd)_i > 0\\
0\:, & \text{otherwise}
\end{cases}\:,
\quad 
(\Delta\bd_-)_i:=
\begin{cases}
(\by^{\hat{k}})_i\:,& \text{if } (\bC\bx^{\hat{k}} - \bd)_i < 0\\
0\:, & \text{otherwise}
\end{cases} \:,
\]
that makes $\bx^{\hat{k}}$ feasible for the relaxed 
problem.\footnote{Alternatively, $\bx^{\hat{k}}$ 
is also feasible for
the adjusted (rather than relaxed) problem obtained by still replacing 
$\bb$ with $\bb'$ but then including instead the {\em equality} constraints
$\bC\bx=\bd'$, with $\bd'=\bC\bx^{\hat{k}}$.}

\section{Implementation and Numerical Experiments}
\label{sec:opt_num_results}

\subsection{A Penalty-Parameter Updating Rule}
\label{subsec:penalty update}

The following updating rule for the penalty parameter was used in
our experiments; here $\sigma_1>0$, $\sigma_2>1$, and $\gamma_0>0$,
$\gamma_1>0$, $\gamma_2>0$, $\gamma_3>0$ are prescribed, but $\gamma_1$
through $\gamma_3$ can be freely reduced with every
increase of $\varphi_k$.

\noindent
{\bf Penalty-parameter updating rule}
\begin{enumerate}
\item Set $\varphi^+:=\varphi$.
\item If $\|(\bz,\by)\|>\gamma_0\varphi$,
set $\varphi^+:=\frac{\sigma_2}{\gamma_0}\|(\bz,\by)\|$.
\item If $\varphi^+\leq\|[\dvi;\bseta-\bszeta]\|_\infty + \sigma_1$,
$\|G_1\|\leq\gamma_1$, $\|G_2\|\leq\gamma_2$, and $|G_3|\leq\gamma_3$,\\
then set $\varphi^+:=\sigma_2(\|[\dvi;{\bseta-\bszeta}]\|_\infty+\sigma_1)$.
\end{enumerate}

We now show that this proposed penalty-parameter updating rule satisfies 
\crefrange{reqP:varphi_Eq}{reqP:phi_increase_Eq}.

\noindent
\ref{reqP:varphi_Eq}:
Clear, since $\varphi_0>0$ in the
Master Algorithm, and $\sigma_1>0$ and $\sigma_2>1$ here. \\
\ref{reqP:boundedz_Eq}:
Step~2 above implies that 
$\varphi_{k+1}\geq\frac{1}{\gamma_0}\|(\bz^k,\by^k)\|$ for all $k$, proving the claim. \\
\ref{reqP:large_enough_penalty_Eq}:
Suppose $\varphi_k=\hat\varphi$ for all $k>\hat k$.
Then, in view of step~2 above, 
it must be the case that 
$\{(\bz^k,\by^k)\}$ is bounded.  
Further, since
$G_1^k$, $G_2^k$, and $G_3^k$ all tend to zero
(see \cref{reqP:large_enough_penalty_Eq}),
step~3 above implies that 
$\hat\varphi>\|[\dvik;\bseta^k-\bszeta^k]\|_\infty+\sigma_1$
for $k>\hat k$, so the requirement is satisfied. \\
\ref{reqP:phi_increase_Eq}:
Suppose $\varphi_k\to\infty$ and $\{(\bz^k,\by^k)\}$ is bounded,
so the condition in step~2 above
cannot hold more than finitely many times.  
Then, since $\varphi_k \to\infty$,
the conditions in step~3 must be satisfied on an infinite index 
set $K$, implying that $G_1^k$, $G_2^k$ and $G_3^k$
are all bounded on $K$, and
$\varphi_k\leq\|[\dvik;\bseta^k-\bszeta^k]\|_\infty + \sigma_1$ on $K$
so that $\|[\dvik;\bseta^k-\bszeta^k]\|_\infty\to\infty$ on $K$.
Thus $\frac{\varphi_k}{\|[\dvik;\bseta^k-\bszeta^k]\|_\infty}\leq 1 
+ \frac{\sigma_1}{\|[\dvik;\bseta^k-\bszeta^k]\|_\infty}$ is bounded on $K$,
proving the claim.

\subsection{Implementation Details}
\label{subsec:implementation}
All numerical tests were run with a Matlab implementation of the Master Algorithm (section~\ref{ProposedAlgorithmStructure}), base iteration (CR-MPC proposed in \cite{LaiuTits19}), and penalty-parameter update (section~\ref{subsec:penalty update}) on a machine with AMD Opteron(tm) CPU Processor 6376 (2.3GHz) and Matlab~R2019a in Linux platform. 

\paragraph{Stopping criterion}
In the implementation,
the stopping criterion for the Master Algorithm was 
$\texttt{Err}\leq\texttt{tol}$
with the normalized error term%
\footnote{%
Approximate primal feasibility 
$(\bs,\bt_+,\bt_-)\geq\bzero$ (or $\approx\bzero$) is implicitly 
taken into account in the last three terms in the numerator.}
\begin{equation}
\texttt{Err}:=\frac{\left\|\left[\bH\bx+\bc-\bA^T\bspi-\bC^T(\bseta-\bszeta);\:\min\{\bs,\bspi\};
	\:\min\{[\bt_+;\bt_-],[\bseta;\bszeta]\}\right]\right\|_2}
{\max\{\|\bH\|_\infty,\,\|\bc\|_\infty\:,\|{\bA}\|_\infty,\,\|{\bC}\|_\infty\}}\:,
\label{eq:Exlambda}
\end{equation}
where $\bs:=\bA\bx-\bb$, $\bt_+:=\bC\bx-\bd$, 
$\bt_-:=-(\bC\bx-\bd)$, and $\min\{\cdot,\cdot\}$ denotes 
component-wise minimum.
When equality constraints are not present, $\texttt{Err}$ is reduced by setting 
$\min\{[\bt_+;\bt_-],[\bseta;\bszeta]\} = \bzero$,
and $\bC = \bzero$.

\paragraph{Initialization}
The Master Algorithm requires that 
$(\bx^0,\bz^0,\by^0)$ be feasible 
for the augmented primal problem \eqref{eq:primal_cqp_phi_rho},
while the CR-MPC base iteration requires 
primal-strictly-feasible initial points, i.e.,
$\bz^0> -\min\{\bA\bx^0-\bb,\bzero\}$
and 
$\by^0> \text{abs}(\bC\bx^0-\bd)$, with
$\text{abs}(\cdot)$ the component-wise absolute value.
In our tests, given a (problem dependent) $\bx^0$,
we chose $\bz^0= c_{\bz} \bone$%
\footnote{A possibility would 
be to freeze $\bz$ at zero when $\bx^0$ is primal feasible and indeed,
when a component $z_i^k$ of $\bz^k$ reaches zero at some iteration $k$,
freeze that component to zero thereafter.  This was not done in the 
tests reported here.}
and $\by^0= c_{\by}\bone$ with $c_{\bz} = -\min\{\min\{\bA\bx^0-\bb\},0\}+1$ and $c_{\by} = \max\{\text{abs}(\bC\bx^0-\bd)\}+1$.
For the initial dual variable and penalty parameter, we used $\dv^0=\bone$ and $\varphi_0=1$.

\paragraph{Base iteration}
The base iteration used in the tests is that of
Algorithm~CR-MPC proposed in \cite{LaiuTits19},
with the stopping criterion turned off, and with implementation
details (including parameter values) essentially identical to those laid 
out in section~3.2 of that paper.  A notable exception is that, here, 
in connection with relaxation variables $(\bz,\by)$, the constraints are 
structurally sparse, and this was specifically attended to in the solution
of the Newton-KKT systems; thus, the associated CPU cost was only slightly 
higher than if there were only $n$, rather than $n+m+p$ variables.
A few constraint selection rules were considered in \cite{LaiuTits19} for Algorithm~CR-MPC.
Here we used {\rulename} with the same parameter values as 
in \cite{LaiuTits19} but with two minor modifications: (i) we keep the slack 
threshold $\delta_k$ equal to its initial value $\bar{\delta}$ in the first five iteration,
and (ii) we always include the sparse constraints $\bz\geq\bzero$, $\bC\bx+\by\geq\bd$, and $\bC\bx-\by\leq\bd$ in the selected constraint set.
In the numerical tests, (i) improved the robustness of {\rulename} to the choice 
of $(\bx^0,\bz^0,\by^0)$, while (ii) led to faster convergence with little additional cost per iteration.

\paragraph{Penalty-parameter update}
We implemented the penalty-parameter update in Master Algorithm following the 
rule given in section~\ref{subsec:penalty update}, 
with parameter values $\sigma_1 = 1$, $\sigma_2=10$,
$\gamma_0:=\frac{\|[\bz^0;\by^0]\|_\infty}{\varphi_0}$ and, for $i=1,2,3$,
$\gamma_i:=\|G_i^0\|_2$.
We chose $\|(\bz,\by)\|:=\|[\bz;\by]\|_\infty$ in step 2, and $\|G_i\|=\|G_i\|_2$ for $i=1,2$, in step 3.
Importantly, at every increase of $\varphi$, 
the internal base iteration variables
(denoted $\texttt{Var}_{\texttt{BI}}$ in the Master
Algorithm in section~\ref{ProposedAlgorithmStructure})
were
reset to the initial values specified in~\cite{LaiuTits19}, since a new
optimization problem (different objective function) is then dealt with.

\paragraph{Detection of infeasibility}
As discussed in section~\ref{sec:infeasibility}, the proposed framework 
provides an infeasible certificate whenever \eqref{eq:primal_cqp} is infeasible.
Stopping criterion~\eqref{eq:Exlambda} was thus augmented with an
alternative criterion (see \eqref{eq:infeasible_condition}) which is
declared satisfied when a ``certificate'' $(\hat{\bspi}^k, \hat{\bsomega}^k)$
is produced such that
\begin{equation}
\label{eq:infeasible_stop}
\bb^T\hat{\bspi}^k + \bd^T\hat\bsomega^k> \sqrt{{\upepsilon}_{\texttt{m}}} \quand 
\frac{\|[\bA^T\hat{\bspi}+\bC^T\hat{\bsomega};\:\min\{\hat{\bspi},\bzero\}]\|_2}
{\max\{\|{\bA}\|_\infty,\,\|{\bC}\|_\infty\}}\leq\texttt{tol}_{\texttt{infeas}}\:,
\end{equation} 
where ${\upepsilon}_{\texttt{m}}$ is the machine precision 
and $\texttt{tol}_{\texttt{infeas}}$ a tolerance parameter,
in which case \eqref{eq:primal_cqp} is declared to be infeasible.

\cref{thm:infeasible_certificate} suggests that $(\hat{\bspi}^k, \hat{\bsomega}^k)$ could be chosen as $(\sfrac{\bspi^k}{\varphi_k}; \sfrac{(\bseta^k-\bszeta^k)}{\varphi_k})$ with $({\bspi^k}, {\bseta^k}, {\bszeta^k})$ dual variables given by the base iteration.
However, we found that for some infeasible problems, this choice requires many iterations to satisfy \eqref{eq:infeasible_stop}. 
In our implementation, 
we constructed $(\hat{\bspi}^k, \hat{\bsomega}^k)$ by setting
\begin{equation}\label{eq:infeasible_dual}
[\hat{\bspi}_{Q}^k; \hat{\bsomega}^k]:=[\,[\bp_{\bspi_{Q}}]_+;\bp_{\bsomega}]\:,\, [\bp_{\bspi_{Q}};\bp_{\bsomega}]:=\text{proj}_{\cN([\bA_{Q}^T, \bC^T])}
\big([\sfrac{\bspi_{Q}^k}{\varphi_k}; \sfrac{(\bseta^k-\bszeta^k)}{\varphi_k}]\big)\:,
\end{equation}
and $\hat{\bspi}_{Q^c} = \bzero$, where $Q$ and $Q^c$ denote the reduced constraint index set and its complement, both given by the CR-MPC base iteration, $\bspi_Q$ and $\bA_Q$ denote the subvector and submatrix of $\bspi$ and $\bA$ associated to the index set $Q$, respectively (see, e.g., \cite{LaiuTits19}, for details).
When such $Q$ is not available, $Q:=\{1,\dots,m\}$ and 
$Q^c=\emptyset$ is appropriate (but not as efficient).
In \eqref{eq:infeasible_dual}, $[\cdot]_+:=\max\{\cdot,\bzero\}$ and $\text{proj}_{\cN([\bA_{Q}^T, \bC^T])}$ denotes the orthogonal
projection operator onto the null space of $[\bA_{Q}^T, \bC^T]$.
We note that, with this choice of $(\hat{\bspi}^k, \hat{\bsomega}^k)$
substituted for $(\overline\bspi^k,\overline\bsomega^k)$, 
\cref{lem:infeas} still holds, so that, on an infeasible problem, \eqref{eq:infeasible_stop}
will eventually be satisfied, and an infeasibility certificate will be 
produced.  In addition, it is intuitively clear, and was verified in our
numerical tests, that this choice results in a much smaller number of 
necessary iterations for~\eqref{eq:infeasible_stop} to be satisfied.
Furthermore, the computational cost of running the infeasibility test
is negligible in comparison with the overall cost of an iteration.
Finally, we set $\texttt{tol}_{\texttt{infeas}}=10^{-6}$ in all numerical tests.

\subsection{Randomly Generated Problems}
\label{subsec:random_problems}

We first tested the Master Algorithm with the CR-MPC base iteration 
on imbalanced ($m\gg n$) randomly generated problems both with and 
without equality constraints.
We considered problems of the form~\eqref{eq:primal_cqp} with 
sizes $m:=10\,000$, $n$ ranging from $10$ to $200$, 
and $p={n}/{2}$ or 0.
We solved two sub-classes 
of problems: (i) strongly convex---$\bH$ diagonal and 
positive definite, with random diagonal entries from uniform 
distribution $\cU(0,1)$---and (ii) linear---$\bH=\bzero$.
For each sub-class, $20$ randomly generated problems were solved for each problem size, 
and the results averaged over the $20$ problems were reported.
Consistent results were observed with $\bH\not=\bzero$ 
but det($\bH$)=0.
The entries of $\bA$, $\bC$, and $\bc$ were taken from 
a standard normal distribution $\cN(0,1)$; as for $\bb$ and $\bd$, see
sections~\ref{subsubsec:feasible} and \ref{subsubsec:infeasible}.

\subsubsection{Comparison on Feasible Problems}
\label{subsubsec:feasible}
To guarantee strict feasibility (\cref{as:PD}), we generated 
$\bx^{\textup{feas}}$ and $\bs^{\textup{feas}}$ with i.i.d.~entries
taken from $\cN(0,1)$ and uniform distribution 
$\cU(1,2)$, respectively, 
and then set $\bb:=\bA\bx^{\textup{feas}}-\bs^{\textup{feas}}$ and 
$\bd:=\bC\bx^{\textup{feas}}$.
{\em For feasible-start algorithms considered in the comparison, the 
starting point was $\bx^0 := \bx^{\textup{feas}}$}, while for the proposed 
infeasible-start (IS) 
framework, a starting point $\bx^0$ was generated by repeatedly taking 
i.i.d.~entries from $\cN(0,1)$ 
	until $\bx^0$ became infeasible.
For scaling purpose, we followed the heuristic proposed in \cite{JOT-12} 
and used the normalized constraints $(\bD_1\bA)\bx\geq \bD_1 \bb$ and 
$(\bD_2\bC)\bx = \bD_2\bd$, where $\bD_1=\diag{(1/\|\ba_i\|_2)}$ and 
$\bD_2=\diag{(1/\|\bc_i\|_2)}$ with $\ba_i$ and $\bc_i$ the $i$-th row 
of $\bA$ and $\bC$, respectively.
The modified $\bA$ and $\bC$
matrices were also used in the stopping criteria \eqref{eq:Exlambda} and \eqref{eq:infeasible_stop}.

\cref{fig:Random_Eq} reports the iteration counts and computation time of the tested 
algorithms on the two sub-classes of problems with equality constraints.
Here the proposed IS framework with the
CR-MPC base iteration (with ``Rule R'' for constraint selection),
IS-CR-MPC, is 
compared to the same with constraint reduction turned off (IS-MPC$^*$).%
\footnote{Here we denote the CR-MPC algorithm with constraint reduction turned off as MPC$^*$ (rather than MPC) to avoid confusion with the original MPC algorithm in \cite{Mehrotra-1992}.}
For both cases, the tolerance $\texttt{tol}$ was set to $10^{-8}$.
Also, three widely used solvers, SDPT3~\cite{Toh-Todd-Tutuncu-1999,Tutuncu-Toh-Todd-2003}, SeDuMi~\cite{Sturm-1999}, and MOSEK (ver.~$9.1.9$)~\cite{andersen2000mosek,Andersen2003} are included in the comparison.%
\footnote{We note that these solvers can solve a broader 
class of problems (e.g., second-order cone optimization, semidefinite optimization) 
than the proposed IS framework. We include them here since they
allow a close comparison with our code, as Matlab implementations of 
SDPT3 and SeDuMi are freely available within
the CVX Matlab package~\cite{cvx,Grant-Boyd-2008} 
and MOSEK provides a convenient Matlab API.
We consider these results as benchmarks for the IS framework.}
In all reported tests, the convergence tolerances for these 
three solvers were set to $10^{-8}$ as well.%
\footnote{To avoid biases due to different
stopping criteria, we also experimented with tolerances set to $10^{-6}$ 
(while keeping $\texttt{tol} = 10^{-8}$ for all MPC versions) and
observed results with a couple fewer iterations
and nearly identical computation time.}
As seen in \cref{fig:Random_Eq}, on such imbalanced CQPs, 
in spite of the fact that, in terms of iteration count, 
IS-CR-MPC is inferior to MOSEK and comparable to SeDuMi, 
the total computation time recorded by IS-CR-MPC
is three to nine times lower than that 
recorded by the second fastest solver (MOSEK).

\begin{figure}[ht]
	\centering
	\captionsetup[subfigure]{labelformat=parens}	
	\subfloat[Strongly convex QP: $H\succ\bzero$]{
		\captionsetup[subfigure]{labelformat=empty, position=top}
		\subfloat[Iteration count]{\includegraphics[width=0.30\linewidth]{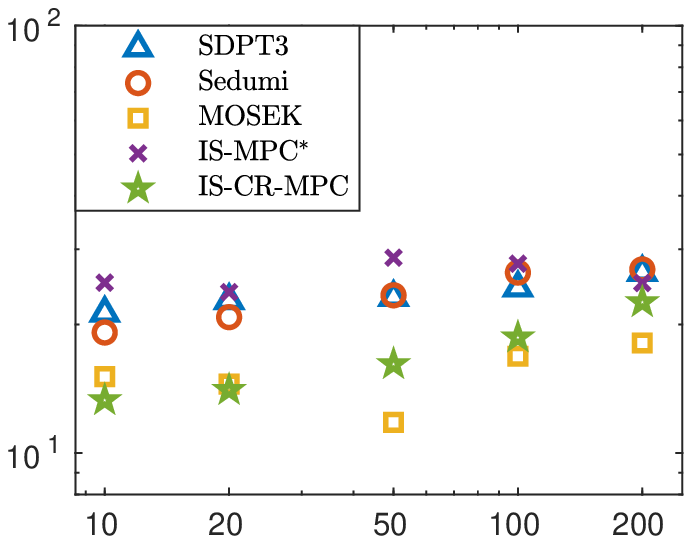}}~~~
		\subfloat[Computation time (sec)]{\includegraphics[width=0.30\linewidth]{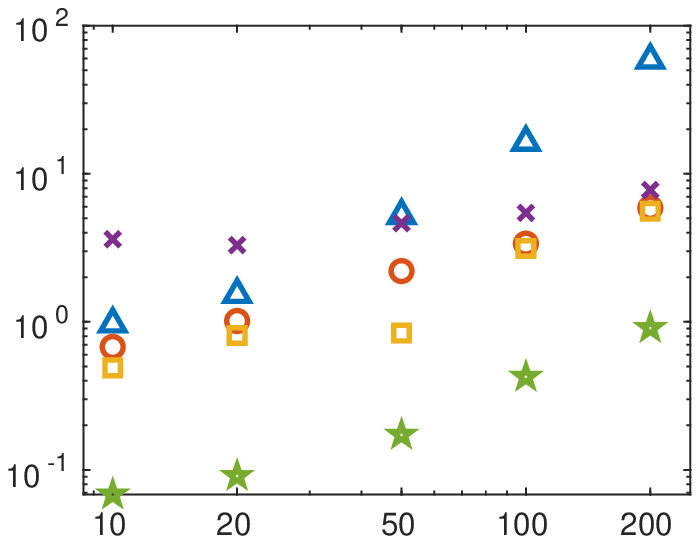}}
		\setcounter{subfigure}{1}
		\label{fig:PD_Eq}}\\	
	\subfloat[linear optimization: $H=\bzero$]{
		\captionsetup[subfigure]{labelformat=empty, position=top}
		\subfloat[Iteration count]{\includegraphics[width=0.30\linewidth]{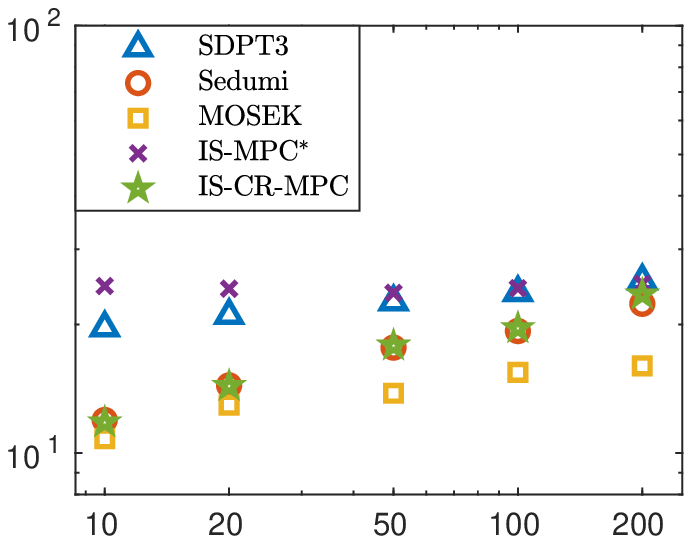}}~~~
		\subfloat[Computation time (sec)]{\includegraphics[width=0.30\linewidth]{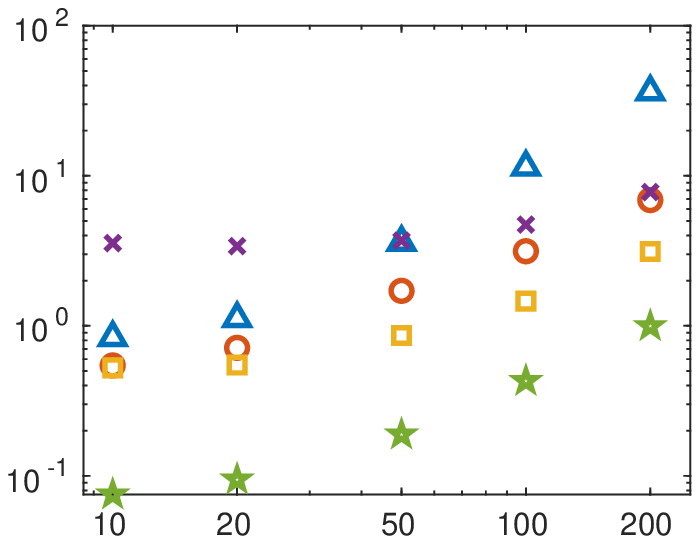}}
		\setcounter{subfigure}{2}
		\label{fig:LP_Eq}}
	\caption{\small Randomly generated problems 
		with $m=10\,000$ inequality constraints and $p=\frac{n}{2}$ equality constraints.
		Results are pictured for two sub-classes of problems. In each figure, the $x$-axis is the 
		number of variables ($n$) and the $y$-axis is the iteration 
		count or total computation time,
		both averaged over the 20 problem instances and plotted in logarithmic scale.}
	\label{fig:Random_Eq}
\end{figure}

\cref{fig:Random} illustrates the results on 
problems with no equality constraints. 
For these tests, we also included the 
feasible-start CR-MPC algorithm of \cite{LaiuTits19} 
and the same with constraint reduction turned off (MPC$^*$)
into the comparison, 
with convergence criterion given in \cite{LaiuTits19} and tolerance $10^{-8}$.
In the linear $(\bH=\bzero)$ case, we included in the 
comparison a revised primal simplex with partial pricing
(see~\cite{BertsimasTsitsiklis97} and references therein) code 
used in \cite{Winternitz-Thesis}%
\footnote{We used an implementation due to Luke Winternitz, 
who kindly made it available to us.}
which takes a two-phase approach: solve an auxiliary problem in phase~1 
to generate a feasible point, then solve the original problem 
from that point in phase~2.

As shown in \cref{fig:Random}, 
the feasible-start MPC$^*$ and CR-MPC solvers
required fewer iterations to 
converge than IS-CR-MPC, 
most likely due to the readily available 
feasible initial point 
(a ``warm-start'' of sorts).
The simplex code required many more iterations than the other 
solvers, but simplex iterations are inexpensive, 
resulting in an average computation time.
On the tested (imbalanced)
problems, the constraint-reduced solvers generally 
outperformed other solvers in terms of computation time. 
The feasible-start CR-MPC algorithm is at most two times faster 
than IS-CR-MPC,
which reflects the difference in iteration counts.
In tests not reported here, we also observed that, when 
starting from the feasible 
$\bx^{\textup{feas}}$, 
IS-CR-MPC and CR-MPC
give nearly identical performance, i.e., the overhead for allowing 
infeasible start is minor.

\begin{figure}[ht]
	\centering
	\captionsetup[subfigure]{labelformat=parens}	
	\subfloat[Strongly convex QP: $H\succ\bzero$]{
		\captionsetup[subfigure]{labelformat=empty, position=top}
		\subfloat[Iteration count]{\includegraphics[width=0.30\linewidth]{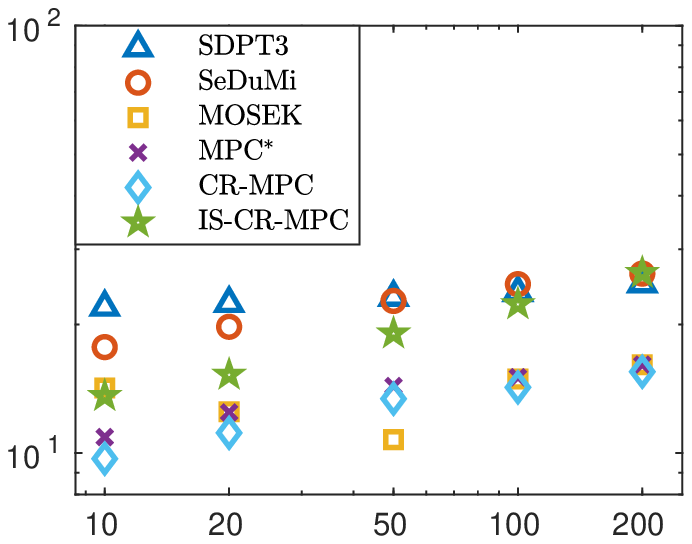}}~~~
		\subfloat[Computation time (sec)]{\includegraphics[width=0.30\linewidth]{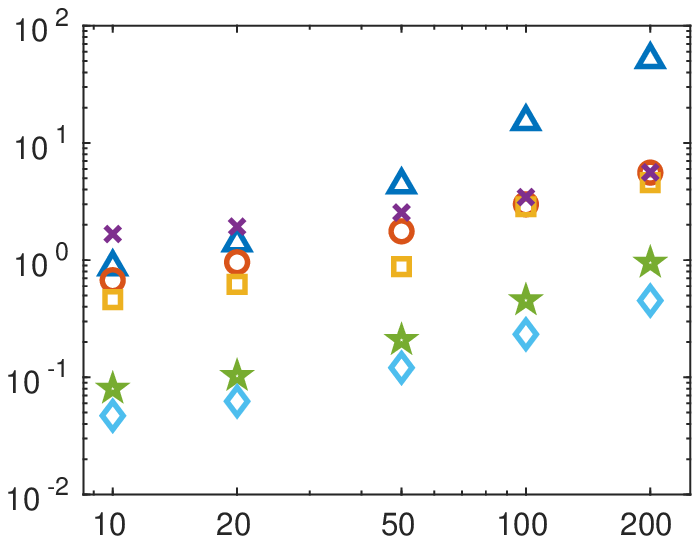}}
		\setcounter{subfigure}{1}
		\label{fig:PD}}\\	
	\subfloat[linear optimization: $H=\bzero$]{
		\captionsetup[subfigure]{labelformat=empty, position=top}
		\subfloat[Iteration count]{\includegraphics[width=0.30\linewidth]{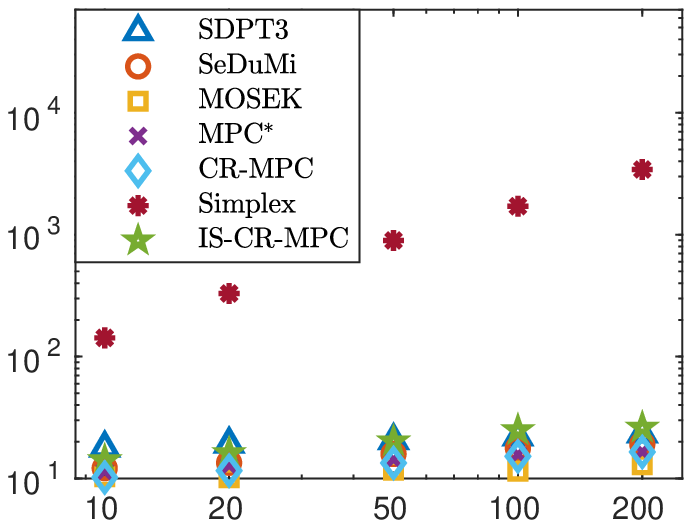}}~~~		\subfloat[Iteration count (zoomed in)]{\includegraphics[width=0.29\linewidth]{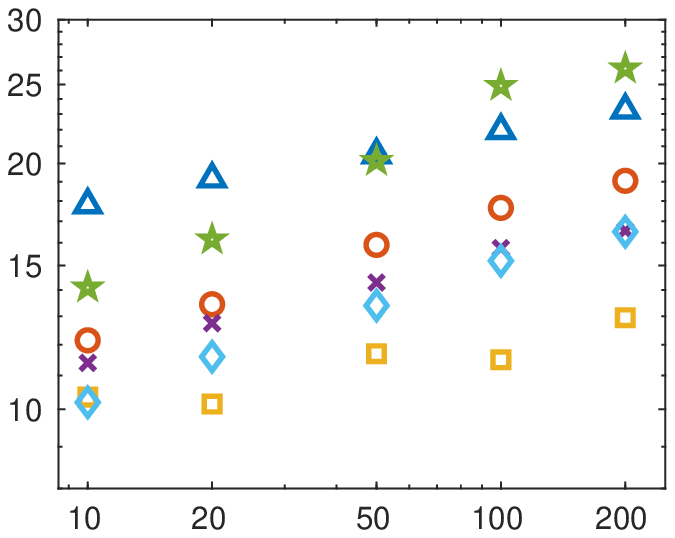}}~~~		
		\subfloat[Computation time (sec)]{\includegraphics[width=0.30\linewidth]{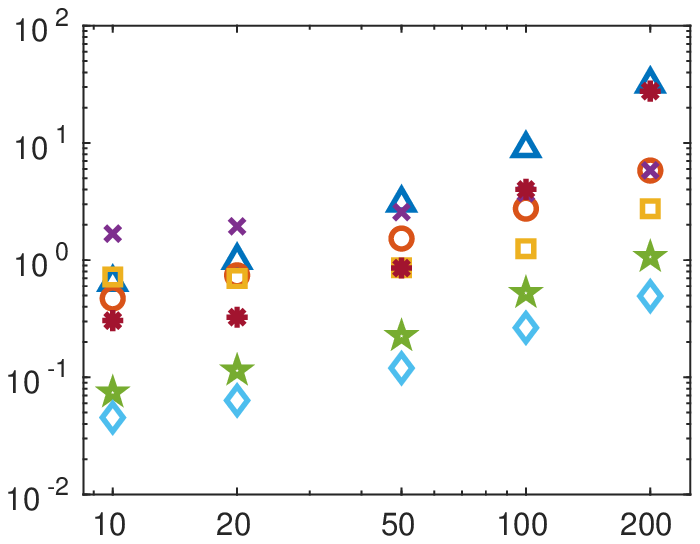}}
		\setcounter{subfigure}{2}
		\label{fig:LP}}
	\caption{\small Randomly generated problems 
		with $m=10\,000$ inequality constraints and no equality constraints.
		Results are pictured for two sub-classes of problems. In each figure, the $x$-axis is the 
		number of variables ($n$) and the $y$-axis is the iteration count or total computation time,		
		both averaged over the 20 problem instances and plotted in logarithmic scale.}
	\label{fig:Random}
\end{figure}

\subsubsection{Infeasibility Detection Tests}
\label{subsubsec:infeasible}
Here, the entries of $\bA,\bb,\bC,\bd$, and $\bc$ were first all generated
from $\cN(0,1)$ (i.i.d.).
To guarantee infeasibility of the problem,  
the last inequality constraint $\ba_m^T\bx\geq b_m$ was then replaced
by $-\ba_i^T\bx\geq -b_i+\delta$, for some index $i$ randomly 
selected from $\{1,\dots,m-1\}$ and $\delta>0$ taken from $\cU(0,1)$.
The starting point $\bx^0$ was generated by taking i.i.d.~entries 
from $\cN(0,1)$.

In \cref{table:Infeasibility_detection}, the averaged iteration counts and 
computation time over the 20 problem instances are reported for IS-CR-MPC.
As seen from the table, with the dual estimates generated 
by \eqref{eq:infeasible_dual}, the 
conditions in \eqref{eq:infeasible_stop} were satisfied 
on all tested problem instances within about 10 iterations on average.
These results suggest that the proposed IS framework is 
capable of providing infeasibility certificates 
efficiently for infeasible problems.
It also is worth noting that no infeasibility certificates were issued
in the tests reported in section~\ref{subsubsec:feasible}.
(i.e., there were no false positives).
\begin{table}[ht]
\small
	\begin{center}
		\begin{tabular}{c|c||rr|rr|rr|rr|rr}
\multicolumn{2}{c||}{$n$}  & \multicolumn{2}{c|}{10} & \multicolumn{2}{c|}{20} &  \multicolumn{2}{c|}{50} &  \multicolumn{2}{c|}{100} & \multicolumn{2}{c}{200}\\
\hline
$p$ & $\bH$     &{Iter.}&{Time}&{Iter.}&{Time}&{Iter.}&{Time}&{Iter.}&{Time}&{Iter.}&{Time}\\
\hline
$\frac{n}{2}$ & $\succ\bzero$   &  7.5 & 0.04 & 11.1 & 0.10 & 10.0 & 0.15 & 10.2 & 0.35 & 10.0 & 1.01 \\
$\frac{n}{2}$ & $=\bzero$	   &  8.5 & 0.05 & 10.4 & 0.08 & 9.9  & 0.15 & 10.1 & 0.38 & 10.1 & 1.08\\
$0$ & $\succ\bzero$   & 10.2 & 0.07 & 11.2 & 0.10 & 10.4 & 0.17 & 10.0 & 0.34 & 10.3 & 0.97 \\
$0$ & $=\bzero$	   & 10.8 & 0.07 & 10.4 & 0.09 & 9.8  & 0.16 & 10.2 & 0.35 & 10.3 & 0.95
		\end{tabular}
	\end{center}
\caption{\small Infeasibility detection results with IS-CR-MPC
on randomly generated
    (infeasible) problems with $m=10\,000$ inequality constraints.
	In each row, the averaged iteration count and computation time (sec) are reported for problems with $n=10,\dots,200$ variables and $\frac{n}{2}$ or $0$ equality constraints,
	in the strongly convex ($\bH\succ\bzero$) or linear ($\bH=\bzero$) sub-classes.}
	\label{table:Infeasibility_detection}
\end{table}

\subsection{Support-Vector Machine Training Problems}
\label{subsec:SVM}
We tested IS-CR-MPC on CQPs 
arising in the training of support-vector machine (SVM) classifiers
for pattern recognition tasks in high dimensions (see, e.g., \cite{JOT-08} 
and references therein for relevant discussions).
In the problems considered here, the training data 
set takes the form $(\bP,\boldsymbol{\ell})$, where 
$\bP\in\bbR^{\bar{m}\times\bar{n}}$, $\boldsymbol{\ell}\in\bbR^{\bar{m}}$
and, for $i=1,\ldots,\bar{m}$, $\bp_i^T$ ($i$-th row of $\bP$) 
denotes a pattern that 
corresponds to a known label $\ell_i\in\{-1,\,1\}$.
The training process of SVMs aims at finding an optimal 
separating hyperplane (when one exists) in the pattern 
space, that separates the ``+'' class patterns (with label 
$\ell_i=1$) from the ``$-$'' class patterns 
(with label $\ell_i=-1$) 
and is equidistant from both classes.
Specifically, the goal is to construct a
hyperplane 
\begin{equation}
\{\bp\in\bbR^{\bar{n}} \colon \vint{\bw,\bp} - \beta = 0 \},
\end{equation}
under
inner product $\vint{\cdot,\cdot}$, such that
the parameters $\bw\in\bbR^{\bar{n}}$ and $\beta\in\bbR$ satisfy
\begin{equation}
\text{sign}\{\vint{\bw,\bp_i} - \beta\} = \ell_i\:,\quad i = 1,\dots,\bar{m}\:,
\end{equation}
while maximizing the separation margin $\frac{2}{\|\bw\|_2}$.
When the Euclidean inner product is selected, this amounts
to solving
\begin{equation}\label{eq:SVM}
\minimize_{\bw\in\bbR^{\bar{n}}, \beta\in \bbR} ~\: \frac{1}{2}\|\bw\|_2^2\quad
\mbox{s.t.} \:  ~   \bL({\bP}{\bw}-\beta\bone)\geq \bone\:,
\end{equation}
where $\bL:=\diag(\boldsymbol{\ell})$.
By denoting $\bx=[\bw;\beta]$, this problem takes the form 
of \eqref{eq:primal_cqp} with $n=\bar{n}+1$ optimization variables and 
$m = \bar{m}$ inequality constraints.
Of course, when the given training data is not separable, \eqref{eq:SVM} is infeasible.
When this is known (e.g., an infeasibility certificate has been
produced by IS-CR-MPC), a constraint-relaxation variable is introduced
that allows misclassification, and the objective function is penalized
accordingly, viz.
\begin{equation}\label{eq:SVM_relaxed}
\minimize_{\bw\in\bbR^{\bar{n}}, \beta\in \bbR, \nu\in \bbR} ~\: \frac{1}{2}\|\bw\|_2^2 \,+ \,\tau \nu
\quad
\mbox{s.t.} \:  ~   \bL({\bP}{\bw}-\beta\bone)+\nu\bone\geq \bone\:,~
\nu\geq0\:,
\end{equation}
where $\tau>0$ is a constant penalty parameter on the relaxation variable $\nu\in\bbR$.
This relaxed problem still takes the form of \eqref{eq:primal_cqp}, with $n=\bar{n}+2$ optimization variables $\bx=[\bw;\beta;\nu]$ and $m = \bar{m}+1$ inequality constraints.%
\footnote{Alternatively, following the suggestion
made at the very end of section~\ref{sec:infeasibility},
an $(n+1)$-variable problem with feasible start could be solved.}

We tested IS-CR-MPC on SVM training 
for four data sets---\texttt{MUSHROOM}, \texttt{ISOLET}, 
\texttt{WAVEFORM}, and \texttt{LETTER}---from the UCI machine 
learning repository \cite{UCI-database}.
As in \cite{Gertz-Griffin-2006}, 
a lifted version of the data, in a higher-dimensional feature space,
with increased likeliness of linear separation was used instead;
see \cite{Gertz-Griffin-2006,JOT-08} for details.
Such mapping results in \texttt{MUSHROOM} and \texttt{ISOLET}
being separable; \texttt{WAVEFORM}, and \texttt{LETTER} are not, and
the relaxed problem \eqref{eq:SVM_relaxed} was solved instead.%
\footnote{A Matlab-formatted version of these data sets was kindly 
made available to us by Jin Jung, first author of~\cite{JOT-08}.}
The numbers of features and patterns for the lifted version of each data set are listed in \cref{table:SVM_size}.
\begin{table}[ht]
	\begin{center}
		\begin{tabular}{ccccc}
			& \texttt{MUSHROOM} &  \texttt{ISOLET} & \texttt{WAVEFORM} & \texttt{LETTER}\\
			\# of features ($\bar{n}$) & 276 & 617 & 	861 & 	153 	\\
			\# of patterns ($\bar{m}$) & 8124 & 7797 & 	5000 & 	20000 	\\
			separable & Yes & Yes & No & No
		\end{tabular}
	\end{center}
	\caption{\small Problem specifications of the four tested data sets for SVM training.}
	\label{table:SVM_size}
\end{table}

The performance of SDPT3, SeDuMi, MOSEK, IS-MPC$^*$, and IS-CR-MPC is
reported in \cref{fig:SVM}, where logarithmic scales are used.
Here the starting point for the IS algorithms was 
$\bx^0 = \bzero$.
The results show that, on these imbalanced CQPs,
IS-CR-MPC enjoys fastest convergence among the tested solvers.
Indeed, compared to the next fastest (MOSEK in all four cases),
the speedups for \texttt{MUSHROOM},\texttt{ISOLET}, \texttt{WAVEFORM},
and \texttt{LETTER} were 1.6x, 4.1x, 1.2x, and 2.3x, respectively.
The lower speed for \texttt{WAVEFORM} is readily explained by the
fact that this data set is the most balanced one among the tested data sets.

\begin{figure}[ht]
	\centering
	\subfloat{
		\captionsetup[subfigure]{labelformat=empty, position=top}
		\subfloat[Iteration count]{\includegraphics[width=0.45\linewidth]{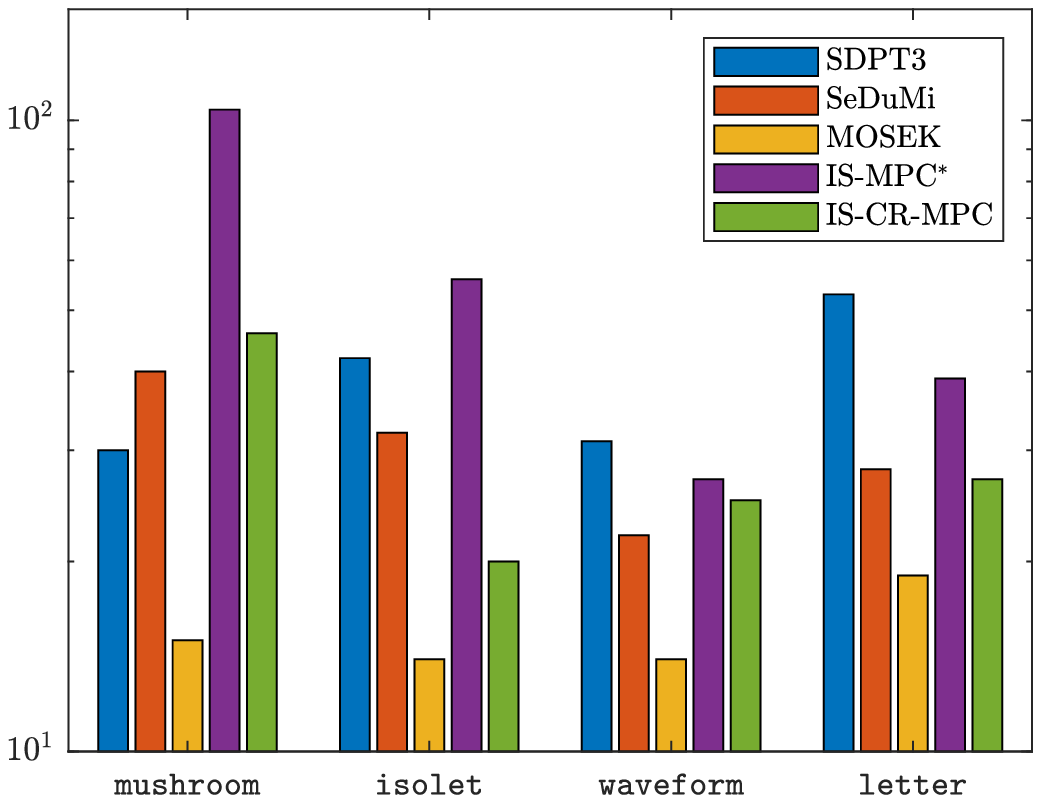}}~~~
		\subfloat[Computation time (sec)]{\includegraphics[width=0.45\linewidth]{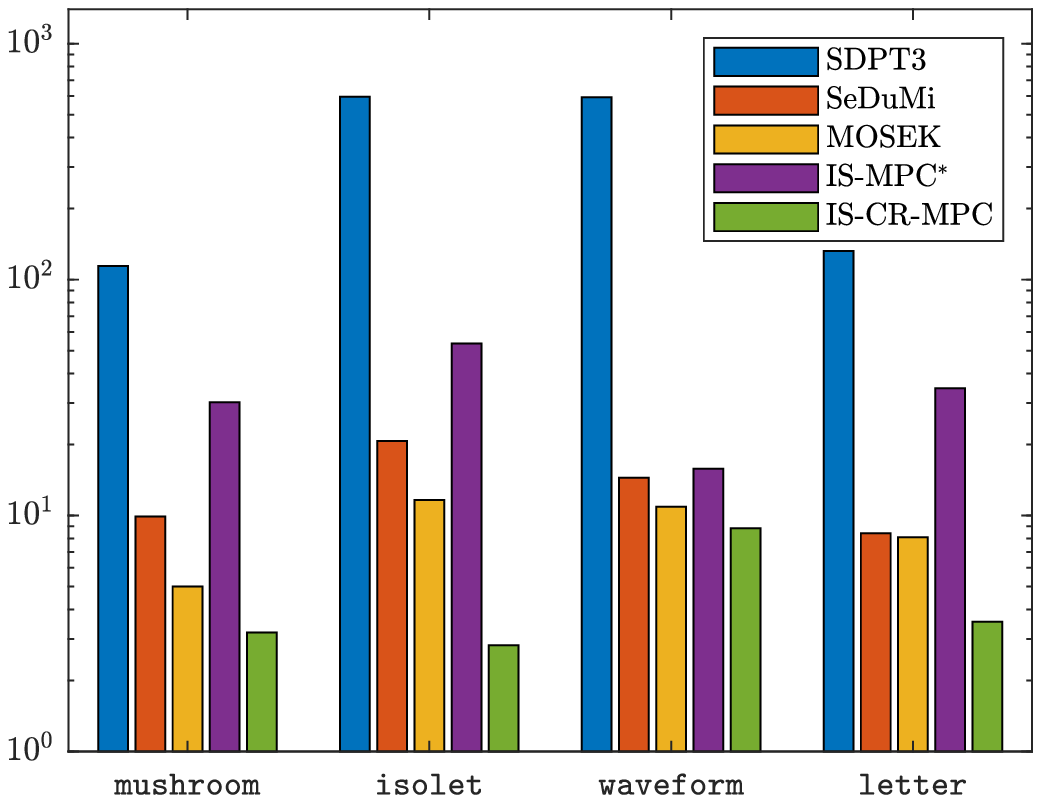}}}\\	
	\caption{\small Support-vector machine training problems.
		Numerical results of tested algorithms on the \texttt{MUSHROOM}, \texttt{ISOLET}, \texttt{WAVEFORM}, and \texttt{LETTER} data sets. In the figures, iteration counts and computation time are reported and plotted in logarithmic scale.}
	\label{fig:SVM}
\end{figure}

\section{Conclusion}
\label{sec:conclusion}

An exact-penalty-based framework for allowing for infeasible starts
in solving CQPs (including linear optimization problems) was proposed
and analyzed.  With negligible additional computational cost per
iteration, an infeasibility test is included that provides an
infeasibility certificate when the problem at hand is indeed infeasible.
The framework was tested on constrained-reduced MPC.
Numerical results suggest that, on imbalanced CQPs,
infeasible-start CR-MPC is significantly faster than SDPT3, SeDuMi,
and MOSEK.
It is also confirmed that constraint
reduction is very powerful on such problems.

\bibliographystyle{siam}
\bibliography{reference}

\end{document}